\documentclass[a4paper,10pt,reqno]{amsart} %
\usepackage{latexsym,amssymb}
\usepackage{hyperref}
\usepackage{amsmath,amsthm,amsxtra}
\usepackage{amsfonts}
\usepackage[american]{babel}
\usepackage{mathrsfs}
\usepackage{psfrag}
\usepackage{epsfig,inputenc}
\usepackage{graphpap,latexsym,epsf}
\usepackage{color}
\usepackage{amssymb,eucal,paralist,color,enumerate, enumitem}
\usepackage{graphicx}
\newcommand{\R}{\mathbb{R}}
\newcommand{\N}{\mathbb{N}}

\newcommand{\eee}{\`e}

\mathchardef\emptyset="001F

\newtheorem{maintheorem}{Theorem}

\newtheorem{theorem}{Theorem}[section]

\newtheorem{problem}[theorem]{Problem}

\newtheorem{lemma}[theorem]{Lemma}
\newtheorem{remark}[theorem]{Remark}

\newtheorem{proposition}[theorem]{Proposition}

\newtheorem{notation}[theorem]{Notation}
\newtheorem{corollary}[theorem]{Corollary}
\newtheorem{hypothesis}[theorem]{Hypothesis}

\numberwithin{equation}{section}

\newcommand{\eps}{\varepsilon}

\newcommand{\weakto}{\rightharpoonup} %convergenza nella topologia debole dello spazio di Banach
\newcommand{\aein}{\text{a.e.\ in }}

\newcommand{\dualoperator}

\def\calD{{\mathcal D}} \def\calE{{\mathcal E}} \def\calF{{\mathcal F}}
 \def\calH{{\mathcal H}} 
\def\calJ{{\mathcal J}}  
  
\def\calP{{\mathcal P}}

\def\rmD{{\mathrm D}}

\newcommand{\dive}{\mathop{\rm div}}

\def\dd{\;\!\mathrm{d}} % differential for integration
\newcommand{\Gdir}{\Gamma_{\mathrm{D}}}
\newcommand{\Gneu}{\Gamma_{\mathrm{N}}}

\definecolor{ddcyan}{rgb}{0,0.1,0.9}
\definecolor{ddmagenta}{rgb}{0.6,0,0.9}
\definecolor{dmagenta}{rgb}{0.6,0,0.6}
\definecolor{orange}{rgb}{0.6,0.2,0}
\definecolor{vgreen}{rgb}{0.1,0.5,0.2}

\newcommand{\piecewiseConstant}[2]{\overline{#1}_{\kern-1pt#2}}

\newcommand{\foraa}{\text{for a.a. }}

\newcommand{\ue}{u^{\mathrm{e}}}
\newcommand{\ur}{u^{\mathrm{r}}}
\newcommand{\vr}{v^{\mathrm{r}}}

\newcommand{\bur}{\bar{u}^{\mathrm{r}}}

\newcommand{\ueo}{{u}_{\Omega}^{\mathrm{e}}}
\newcommand{\uro}{{u}_{\Omega}^{\mathrm{r}}}

\newcommand{\upme}{{u}_{\Omega}^{\mathrm{e},\pm}}
\newcommand{\ube}{{u}_{B}^{\mathrm{e}}}

\newcommand{\upmr}{{u}_{\Omega}^{\mathrm{r},\pm}}
\newcommand{\ubr}{{u}_{B}^{\mathrm{r}}}
 \def\trait #1 #2 #3 {\vrule width #1pt height #2pt depth #3pt}

 \def\fin{\hfill
         \trait .3 5 0
         \trait 5 .3 0
         \kern-5pt
         \trait 5 5 -4.7
         \trait 0.3 5 0
 \medskip}
 
\newcommand{\QED}{\mbox{}\hfill\rule{5pt}{5pt}\medskip\par}

\newcommand{\param}{(\eps, \lambda, \mu, b,\rho)}
\newcommand{\bparam}{(\bar\eps, \bar\lambda, \bar\mu, \bar b,\bar\rho)}
\newcommand{\paraml}{(\bar{\lambda},\bar{ \mu}, \bar{b},\bar{\rho})}
\newcommand{\paramn}[1]{(\eps_{#1}, \lambda_{#1}, \mu_{#1}, b_{#1},\rho_{#1})_{#1}}
\newcommand\JUMP[1]{\mathchoice
                  {\big[\hspace*{-.3em}\big[#1\big]\hspace*{-.3em}\big]}
                   {[\hspace*{-.15em}[#1]\hspace*{-.15em}]}
                   {[\![#1]\!]}
                   {[\![#1]\!]}}
 \newcommand{\trp}[1]{{#1}^+}
  \newcommand{\trm}[1]{{#1}^-}

\newcommand{\symprod}{\otimes_{\mathrm{S}}}
\newcommand{\tr}[1]{\mathrm{tr}\left(#1\right)}
\newcommand{\Sym}{\R_{\mathrm{sym}}^{3\times 3}}
\newcommand{\Lin}{\mathrm{Lin}}
\newcommand{\BV}{\mathrm{BV}}
\newcommand{\abv}{\mathsf{X}}

\newcommand{\abop}{\mathsf{A}}
\newcommand{\babop}{\bar{\mathsf{A}}}
\newcommand{\abfo}{\mathsf{F}}
\newcommand{\babfo}{\bar{\mathsf{F}}}
\newcommand{\absp}{\mathsf{H}}
\newcommand{\babsp}{\bar{\mathsf{H}}}
\newcommand{\Spu}{\mathsf{U}}
\newcommand{\Spv}{\mathsf{V}}
\newcommand{\abopn}[1]{\mathsf{A}_{#1}}
\newcommand{\abfon}[1]{\mathsf{F}_{#1}}
\newcommand{\abspn}[1]{\mathsf{H}_{#1}}
\newcommand{\vo}[1]{{#1}_{\Omega}}
\newcommand{\voe}[1]{{#1}_{\Omega}^{\mathrm{e}}}

\newcommand{\vbe}[1]{{#1}_{B}^{\mathrm{e}}}

\newcommand{\vop}[1]{{#1}_{\Omega}'}
\newcommand{\vB}[1]{{#1}_{B}}
\newcommand{\vb}[1]{{#1}_{B}}
\newcommand{\vs}[1]{{#1}_{S}}
\newcommand{\voB}[1]{{#1}_{B}'}
\newcommand{\projn}[1]{\mathsf{P}_{#1}}

\newcommand{\sym}{\mathrm{sym}}
\newcommand{\urn}{u_n^{\mathrm{r}}}
\newcommand{\uen}{u_n^{\mathrm{e}}}

\newenvironment{rcomm}{\color{red}}{\color{black}}

\newenvironment{rnew}{\color{ddmagenta}}{\color{black}}

\newcommand{\berin}{\begin{rnew}}
\newcommand{\erin}{\end{rnew}}

\newcommand{\quiname}{(\eps,\lambda,\mu,b,\rho)}
\newcommand{\qui}[1]{(\eps_{#1},\lambda_{#1},\mu_{#1},b_{#1},\rho_{#1})}
\newcommand{\nq}[1]{q_{#1}}

\newcommand{\beroc}{\begin{rcomm}}
\newcommand{\eroc}{\end{rcomm}}

\newenvironment{newricky}{\color{ddcyan}}{\color{black}}
\newcommand{\bnr}{\begin{newricky}}
\newcommand{\enr}{\end{newricky}}
\newcommand{\sop}[1]{\mathscr{S}_{#1}}
\newcommand{\ubn}{u_{B,n}}
\newcommand{\wubn}{\widehat{u}_{B,n}}
\newcommand{\RRR}{\color{magenta}}

\newcommand{\EEE}{\color{black}}

\newcommand{\RRE}{\color{black}}%{\color{ddmagenta}}

\title[]{Dynamics of two linearly elastic bodies
connected by a heavy thin soft viscoelastic layer}
\author{Elena Bonetti}
\address{E.\ Bonetti, Dipartimento di Matematica, Universit\`a degli studi di Milano, via Saldini 50, 20133 Milano - Italy}
\email{elena.bonetti\,@\,unipv.it}
\author{Giovanna Bonfanti}
\address{G.\ Bonfanti, DICATAM -- Sezione di Matematica, Universit\`a degli studi di Brescia, via Valotti 9, 25133 Brescia - Italy}
\email{giovanna.bonfanti\,@\,unibs.it}
\author{Christian Licht}
\address{C.\ Licht, LMGC, Univ Montpellier, CNRS, Montpellier, France \emph{and} Dept. Maths, Mahidol University, Bangkok, Thailand}
\email{christian.licht\,@\,umontpellier.fr}

\author{Riccarda Rossi}
\address{R.\ Rossi, DIMI, Universit\`a degli studi di Brescia, via Branze 38, 25133 Brescia - Italy}
\email{riccarda.rossi\,@\,unibs.it}

\thanks{
This work was partially carried out during a visit of C.L.\ at the Sezione di Matematica of DICATAM (University of Brescia),  also supported by
 Gruppo Nazionale per  l'Analisi Matematica, la
  Probabilit\`a  e le loro Applicazioni (GNAMPA) of the Istituto Nazionale di Alta Matematica (INdAM)
 }

\begin{document}

%running title:
\begin{abstract}
In this paper we extend the asymptotic analysis in \cite{LLOO},
performed on a structure consisting of two linearly elastic bodies connected by a thin soft nonlinear Kelvin–Voigt viscoelastic adhesive layer,  to the case in which
the total mass of the layer remains strictly positive as its thickness tends to zero.
\par
We obtain convergence results by means of a nonlinear version of Trotter's theory of approximation of semigroups acting on variable Hilbert spaces.
Differently from the limit models derived in \cite{LLOO}, in the present analysis the dynamic effects on the surface on which the layer shrinks do not disappear. Thus,  the limiting behavior of the remaining bodies is described not only in terms of their displacements on the contact surface, but also by an additional variable that keeps track of the dynamics in the adhesive layer.
\EEE
%\PERME
%scrivere noi abstract
%\\
%spiegare nell'introduzione 'heavy thin soft'... \EEE
\end{abstract}

\maketitle
\section{Introduction}
%\RRE qui dovremmo espandere un po' l'introduzione: possiamo mantenere il suo schema, ma dovremmo:
%\begin{itemize}
%\item spiegare un po' piu' in dettaglio il problema
%\item spiegare un po' piu' in dettaglio il nostro metodo, sottolineando
%\item come il ns.\ metodo sia diverso da altre tecniche di riduzione dimensionale in letteratura, cf.\ M\"uller/Mora, Mora/Davoli, Thomas/Mielke/Roubicek, altri lavori di Roubicek etc..
%\item dare una outline del lavoro...
%\end{itemize}
% \EEE
%%%
%\footnote{Dear Christian,
%in typing your notes we have tried to be as faithful as possible to the text you suggested. There have been some minor changes, however, and sometimes we have expanded a little bit arguments or explanations, whenever we had troubles following them.... we would like to write a paper where we are able to understand everything, and where arguments are not too condensed. Whenever we have made the above changes to your text, we have commented them in {\bf \LIC blue footnotes}. Whenever we have encountered small problems, or we have doubts, we have highlighted the text in {\bf \PROB red}, and commented our doubts in  {\bf \PROB red footnotes}.
%\\
%A general remark: We have checked on the dictionary, and it seems that `adherent' is a synonim for 'disciple'... so we'd suggest to replace this
%word by `adhering body'...
%\\
%\PROB Last but not least: Christian, could you please let us know what we should write as your affiliation?}
%%%
PDE systems  coupling bulk and surface equations play an important role in several applications. In particular, they are used to describe different physical  situations in which two spatial scales are involved: a macroscopic scale for the bulk domain and a microscopic scale to capture dynamics on a thin layer located at the boundary. Among  others, models for contact with adhesion between rigid bodies represent an important application of this kind of approach. Indeed, these models  couple mechanical and thermal properties of the involved bodies and of  the microscopic configurations of the
 thin adhesive layer \EEE
% adhesive substance
 %in a thin layer
  between the bodies.
 \par
  In a macroscopic description, this layer is considered as a part of the boundaries and dynamics of the physical variables are described by  boundary  equations. \RRE This feature occurs, e.g., in the models for  contact  with adhesion between  a viscoelastic body and a rigid support
  analyzed in, e.g.,  \cite{BBR1}, \cite{BBR3}, % \cite{BBRfrictemp},
 and  \cite{BBRen}. Such models are derived from  the theory for damage in thermoelastic materials  by  \textsc{Fr\'emond}  \cite{FRE87,FRE01}. Specifically, the related energy functionals and dissipation potentials are written both  in the bulk and on the surface and, accordingly, bulk and surface equations are recovered via a generalization of   the principle of virtual powers. The main idea is to account for  the effects of the microscopic forces, responsible for the degradation of the  adhesion \EEE on the interface between body and support,
 in the energy balance.
While the PDE systems from   \cite{BBR1}, \cite{BBR3}, % \cite{BBRfrictemp},
 and  \cite{BBRen} are \emph{rate-dependent}, delamination can also be treated as a \emph{rate-independent} phenomenon, see e.g.\ \cite{KMR, RSZ}. In that modeling context as well,
 the microscopic damage in the interface is assumed to influence the strength of the adhesion and unilateral conditions are accounted for to ensure non-penetrability between the adhering  bodies.
  \par
  A possible validation of this kind of models, coupling bulk and surface phenomena,
  could be provided by deriving the surface equations  from equations set on a thin layer, as the thickness of the layer tends to zero.
  %consists in the asymptotic analysis to pass from  domain to surface equations.
  This kind of asymptotic analysis has been tackled in the literature using different analytical techniques and modeling approaches.
%The main idea consists in introducing equations for the  the interfaces as the limit of a thin medium.
One possibility is to develop  a formal asymptotic expansion method as in \cite{KLA91,DLR14,LR11}. For damage and delamination, we refer to the asymptotic analyses carried out in  \cite{BBLR1,BBLR2}.
In the context of rate-independent modeling of delamination,  instead,  $\Gamma$-convergence type techniques were used in \cite{mielke} to show that \emph{Energetic solutions} to a system for isotropic damage converge to an Energetic solution of a delamination model as the thickness of the layer between the two bulk bodies,  where damage occurs, tends to zero.
Indeed, the \emph{Energetic} weak solvability notion for rate-independent processes, consisting of an energy-dissipation balance and of a stability condition that involves the minimization of a suitable  functional,  allows for the usage of the variational techniques at the core of the analysis in
 \cite{mielke} (see also \cite{FPRZ, FRZ}).
% In addition, let us recall a different approach introduced, e.g.,   in \ to pass to the limit in a   between three bodies getting to the limit  delamination of two  specimens which are in contact with adhesione through a  glue  of thickness 0. In the limit delamination is described  by an energy functional involving the gradient  of the delamination variable as well as a noninterpenetration condition on the displacements u along the interface. The analysis uses the notion of energy solutions and involves rate-independent evolutions. Concenring the limit case when the thickness of the
%adhesive is 0, we mainly refer to the models of contact with adhesion investigated, e.g., in
\par
 A rigorous approach based on variational convergences techniques has been carried out for this kind of problems, in the \emph{rate-dependent} framework, in a series of papers,
cf.\ e.g.\  \cite{L93, LM97, LLOO, Licht-Weller}. In particular, this paper follows up on the analysis developed in
%In  this paper, we will in fact resort to the approach used in
  \cite{LLOO}, where \EEE  the authors derived an asymptotic model for the dynamics
of two linearly elastic bodies connected by a thin viscoelastic layer
\RRE by means by of a nonlinear version of Trotter's theory of approximation of semigroups,  cf.\ \cite{Trotter}. \EEE
% of operators acting on variable Hilbert spaces.
More specifically, in  \cite{LLOO}
the model was obtained by studying the asymptotic behavior,
as some  parameters accounting for  geometrical
and mechanical data vanished, \EEE
of the structure consisting of the two adhering bodies, perfectly
bonded through the  adhesive layer. \EEE The analysis was carried out  under the further assumption that the total mass of the
adhesive layer was vanishing.  Hence, \EEE  the limit model  obtained in \cite{LLOO} describes
for the dynamics of two  adhering bodies subject to a mechanical constraint along the surface $S$ the layer shrinks to. Its constitutive equation is
of the same type as that for the layer (nonlinear viscoelastic of Kelvin-Voigt type).
\par
In this paper, we aim to
extend the  asymptotic \EEE analysis in  \cite{LLOO}
 by considering the case in which the total mass remains strictly positive;  indeed, this is what the term `heavy' in the title refers to.
 As in  \cite{LLOO},   the cornerstone of our  analysis will be   the reformulation of the
% \PROB genuine\footnote{\PROB  what do we mean by `genuine' problem?}   \RRE
 original
problem,
in which the interface is given with a positive thickness,
  in terms of a nonlinear evolution equation  in a Hilbert space of admissible states with finite mechanical energy,  governed by a suitable maximal monotone operator. Our convergence result shall then be  obtained by resorting to  a nonlinear version of Trotter's theory of approximation of semigroups of operators acting on variable Hilbert spaces, see \cite{ILM09}.
  \EEE
 %\EEE   We first write th then  we reformulate the evolution system in a abstract framework, and finally we pass to the limit using \RNEW Trotter's theory. \EEE
 %the so called variational Trotter convergence. \EEE
  Albeit relying on the same theoretical tools as those of \cite{LLOO}, our analysis here is  significantly different.
 Indeed, \EEE
 since the dynamic effects in the thin layer do not disappear, the limiting contact condition between the two remaining bodies shall not only involve their displacements along the interface but an additional variable, too, which accounts for the asymptotic behavior of the layer  and whose analytical treatment  within Trotter's theory calls for suitable arguments. \EEE
  Of course, such a variable may be eliminated so that the constraint appears as viscoelastic with long memory,  cf.\ Sec.\ \ref{s:5}. \EEE
%\par Our analysis shall follow the lines of \cite{LLOO}. Accordingly, the  \RNEW original \EEE problem shall be  formulated in terms of a nonlinear evolution equation  in a Hilbert space of admissible states with finite mechanical energy, governed by a maximal monotone operator.  The convergence results will be derived via Trotter's theory.
\paragraph{\bf Plan of the paper.}  In Section 2 we specify the setting of the problem, starting from the  formulation of the model when the thickness of the interface is positive. Then, in Section 3 we
\RRE recast this problem as an abstract evolution equation in a  Hilbert space, governed by a suitable maximal monotone operator.
Staying with this formulation, in Section 4 we carry out our asymptotic analysis, as the thickness of the layer between the two adhering bodies vanishes, by means of Trotter's theory. In this way we  prove the main result of this paper, \underline{\textbf{Theorem \ref{thm:main}}}.
%prove that this problem admits a suitably defined weak solution  to be used to perform  to the case in whic(Section 4).
Finally, in Section 5 we give some further comments on our result,     and hint at some  extensions.
\par
Throughout this paper, we will use the following notation.
%\footnote{\LIC In Section 2, establishing the  setup of the problem,   we have kept  some of the old text, since it was not so different from what you suggested us to write.. Likewise, we have kept this paragraph on notation...}
\begin{notation}[General notation]
\label{notation-initial}
\upshape
We will denote the orthonormal basis of $\R^3$ by $(e_1,e_2,e_3)$.  Given a  vector $\xi =(\xi_1,\xi_2,\xi_3) \in \R^3$, we will use the symbol $\widehat\xi$ for $(\xi_1,\xi_2)$, so that we  will often write $(\widehat{\xi},\xi_3)$ in place of $(\xi_1,\xi_2,\xi_3)$.
The symbol $\tr A$  will denote the trace of a $\R^{3\times 3}$ matrix,   $\Sym$ the space of $(3{\times} 3)$-symmetric matrices,
%\footnote{\LIC We suggest to change notation here, it seems to us that $S^3$,
%which you suggested for $\Sym$,  could be mixed up with the surface $S$..}
equipped with the standard inner product,
and $\Lin(\Sym)$ the space of linear mappings from $\Sym$ to $\Sym$.
 Given two vectors $\xi,\, \zeta \in \R^3$, we shall denote by $\xi \symprod \zeta$ their symmetrized tensor product,
defined by
\begin{equation}
\label{def:sym-t-prod}
\xi \symprod \zeta  \text{ is the symmetric $(3{\times} 3)$-matrix  with entries } \frac12 (\xi_i \zeta_j  + \xi_j\zeta_i) \ i,j=1,\ldots, 3.
\end{equation}
%\par
%Given a Banach space $X$, we will denote by $\| \cdot\|_{X}$ its norm, and by $\pairing{}{X}{\cdot}{\cdot}$
% the duality pairing between its dual $X^*$ and $X$. The symbol $\BV([0,T];X)$ shall stand for the space
% of functions  from $[0,T]$ to $X$ with bounded variation, whereas we will denote by $\BV^2([0,T];X)$ the space of functions in $\BV([0,T];X)$
% with distributional derivative belonging to the same space.
 \par
 With any subset $O\subset \R^3$ we will associate its   characteristic function  $\chi_O$, defined by
 $\chi_O(x) =1$ if $x\in O$, and $\chi_O(x)=0$ if $x\in \R^3\setminus O$. % shall stand for the
\par
Finally, throughout the paper,  the symbol $C$ will denote various constants that may differ from one line to the other.
 \end{notation}
 %%%%%
 %%%%%
\section{Setup of the problem}
Let us specify the setup of our problem, namely the study of the dynamic response of a structure made up of two adhering bodies connected by a thin adhesive layer, subject to a given load. First of all, the reference configuration of the structure is a bounded connected open subset
$\Omega \subset \R^3$ with a  Lipschitz boundary $\partial\Omega$.
Hereafter, we will denote by $S$ the  set
%projection on $\R^2$ of $\Omega  \cap (\R^2 {\times} \{0\})$, viz.\
%\footnote{\RRE A voler essere precisi, $S$ scritto in questo modo \`e ancora un sottoinsieme di $\R^3$. Dovremmo dire che $S$ viene identificato con la sua proiezione in $\R^2$....}
\begin{equation}
\label{contact-surf}
 S: =  \Omega \cap
\{ (x_1,x_2,x_3)\in \R^3\, : \ x_3=0 \} \text{ and assume that its Hausdorff measure } \calH^2(S)>0. \EEE
\end{equation}
 In what follows, we shall identify $S$ with its projection onto $\R^2$ and therefore treat it as a subset of $\R^2$. \EEE
%and assu $\calH^2(S)$ is strictly positive.
\begin{notation}
\label{not-for-traces}
\upshape
For a function $u\in H^1(\Omega{\setminus} S; \R^3)$, we will denote by
$\trp{u} $  ($\trm u$, respectively),  its restriction  to the open set
 $\Omega^\pm: = \{x=(x_1,x_2,x_3)\in \Omega \, : \, \pm x_3>0\}$,
%(to
%$\Omega^-: =\{x=(x_1,x_2,x_3) \in \Omega \, : \, x_3<0\}$, resp.),
 which is a function in $H^1(\Omega^\pm)$. % (in  $H^1(\Omega^-)$, resp.).
The symbols $\gamma_{S}(\trp u)$ and  $\gamma_{S}(\trm u)$  will denote the traces of \EEE
$\trp u$  and $\trm u$, respectively,  on the set $S$.
%footnote{\LIC Christian,
%previously we would use the symbols $\trp u$ and $\trm u$ for the traces as well, but then in most of the handwritten notes you have sent to us, you use the symbols  $\gamma_{S^\pm}(u)$...}
Moreover, we will use the notation
\begin{equation}
\label{jumps&means}
\text{jump of $u$ across $S$:} \qquad
\JUMP{u}: = \gamma_{S}(u^+) - \gamma_{S}(u^-)\,.
\end{equation}
\end{notation}
Throughout the paper, we will assume that there exists $\eps_0>0$ %\footnote{Shall we take   $\eps_0=1$  for simplicity?} such that
% ELE: QUI IO METTERE , userei solo $\eps$ come termine
such that
\begin{equation}
\label{Beps0}
B_{\eps_0} : = \{ x = (x_1,x_2,x_3) \in \Omega\, : \ |x_3| <\eps_0 \} \text{ is equal to  } S \times (-\eps_0,\eps_0)\,.
\end{equation}
For $0<\eps<\eps_0$, we will  assume that the adhesive occupies the layer
 $B_\eps : = S \times (-\eps,\eps)$, while the two adhering bodies shall occupy  the sets  $ \Omega_\eps^{\pm}: = \{ x = (x_1,x_2,x_3) \in \Omega\, : \ \pm  x_3>\eps\} $. We let
$
\Omega_\eps : =   \Omega_\eps^+ \cup \Omega_\eps^- = \Omega \setminus \overline{B}_\eps\,.
$
 We will use the notation
\begin{equation}
\label{Beps0-pm}
S_\eps^{\pm}: = \{ x\in \Omega \, : \ x_3 = \pm \eps\}, \qquad B_\eps^+ := S \times (0,\eps), \qquad B_\eps^-: = S\times (-\eps, 0)
\end{equation}
and assume that adhesive and adhering bodies are \emph{perfectly stuck} together along $S_\eps: = S_\eps^+ \cup S_\eps^-$.
This means that the jumps across $S_\eps$ both of the displacement $u$ and of the normal stress $\sigma e_3$ are zero, cf.\ \eqref{PDE-sigma-nojump} and \eqref{PDE-u-nojump} below.
\par
We consider a partition of $  \partial \Omega  =\Gdir\cup\Gneu$
such that $\Gdir$ has positive  two-dimensional \EEE Hausdorff measure  and positive distance from $B_{\eps_0}$; we \EEE  assume that, during the time interval $(0,T)$,
the structure is clamped on $\Gdir$ and subjected to  volumetric and  surface forces (on  $\Gneu: = \partial\Omega \setminus \Gdir$),
with densities $f$ and $g$, respectively. 
  We let
$\Gdir^{\pm}: = \Gdir \cap \{ \pm x_3 >0\}$.
% and suppose that\footnote{non ci serve anche questa ipotesi, che non troviamo nelle sue note?}
%
 % where we denote by $\Gdir$ the Dirichlet part of the boundary, with , and by  the Neumann part of $\partial\Omega$. Hereafter, we will suppose  that
%\[
%\mathcal{H}^2(\Gdir^{+}), \, \mathcal{H}^2(\Gdir^{-})>0 \quad \text{and that} \quad \Gdir \cap \overline{B}_{\eps_0} = \emptyset,
%\]
%where we have used the notation $
%\Gdir^{+}: = \Gdir \cap  (\R^2 {\times} (0,+\infty)) $ and $ \Gdir^{-}: = \Gdir \cap  (\R^2 {\times} (-\infty,0)) $. It follows from the third of \eqref{positive-measures} that
%Therefore,
%\begin{equation}
%\label{Gdir-epsn}
%$\Gdir \subset \overline{\Omega}_{\eps} $  for all $ 0<\eps<\eps_0, $
%\end{equation}
%and  for the same values of $\eps$, $\Gneu \cap  \overline{B}_{\eps} \neq \emptyset$.
\par
The adhering bodies are modeled as linearly elastic materials with a strain energy density
$W$ such that
\[
\begin{gathered}
W(x,e) = \frac12 a(x) e \cdot e  \qquad \foraa x \in \Omega \ \text{ and for all } e \in \Sym, \quad \text{with }
\\
a \in L^\infty (\Omega;\Lin(\Sym)) \text{ such that } \exists\, \alpha,\, \beta>0  \ \foraa x \in \Omega \
 \text{for all } e \in \Sym \, : \quad \alpha |e|^2 \leq a(x) e \cdot e  \leq \beta|e|^2\,.
\end{gathered}
\]
The adhesive is assumed homogeneous, isotropic, and `viscoelastic of Kelvin-Voigt generalized type'. Its strain energy density reads as
\[
  W_{\lambda,\mu} (e) : =  \frac{\lambda}2 |\mathrm{tr}(e)|^2 +
\mu  |e|^2 \quad \text{for all } e \in \Sym, \quad\text{with } \lambda,\, \mu>0\hbox{ the Lam\'e constants}.
\]
We will denote by $\rmD W_{\lambda,\mu}(e)$ its differential at any $e\in \Sym$.
 Observe that % in such a way that there exists $c_{\lambda\mu}>0$ such that
\begin{equation}\label{ellittico}
2 W_{\lambda,\mu} (e) =  \rmD W_{\lambda\mu}(e) \cdot e\geq  2\mu |e|^2 \qquad \text{for all } e \in \Sym.
 \end{equation}
 \EEE

Dissipation in the adhesive is modeled through a dissipation potential $\calD : \Sym \to [0,\infty)$ such that   $\calD$  is convex and fulfils
\begin{equation}\label{dissipation}
\exists\, p\in [1,2],\  \ \exists\, \alpha',\, \beta'>0 \ \ \forall\, e \in \Sym\, : \quad \alpha'|e|^p\leq \calD(e) \leq  \beta' (|e|^p{+}1); \EEE %\footnote{\GGE l'ho numerata perch\`e \`e utile richiamarla dove compare l'esponente $p$ \EEE}
\end{equation}
 we will denote by $\partial \calD : \Sym \rightrightarrows \Sym$ its subdifferential in the sense of convex analysis.
Indeed, in system \eqref{PDE-original} below the functional $\calD$ shall be multiplied by a positive constant $b$
 that accounts for the intensity of viscous effects. \EEE
Finally, we assume that  the density $\gamma$ of the structure
 takes two
different positive values in $\Omega_\eps$ and $B_\eps$,
 %the mass density $\varrho$\footnote{We
%might take $\rho^*$ depending on $x$, but for simplicity for the
%moment we have disregarded this..  \RRE Should we take $\rho^* = \rho^*(x)$??}
namely
\begin{subequations}
\label{density-gamma}
\begin{align}
&
\label{varrho} \gamma(x) = \begin{cases} \rho^*(x)  &  \text{for
a.e. } x \in \Omega_\eps,
\\
\rho & \text{for a.e. } x \in B_\eps,
\end{cases} \qquad \text {with }
\\
&
\label{rho-star}
\rho^*:\Omega \to (0,\infty) \text{ a measurable function s.t. } \quad \exists\,  0 <\bar{\rho}_m<\bar{\rho}_M \ \foraa x \in \Omega\, : \quad
\bar{\rho}_m\leq \rho^*(x) \leq \bar{\rho}_M\,.
\end{align}
\end{subequations}
\par
%Via the Principle of Virtual Power we deduce  that
 Then,   %\footnote{\GGE the Principle of Virtual Power credo che porti alla formulazione debole \EEE }
the model for the dynamic response of the  structure, in the case the thin adhesive layer still  has a  `positive thickness', is described by the
% made up of two adherents connected by a thin adhesive layer
following PDE system.
\begin{subequations}
\label{PDE-original}
\begin{align}
& \label{PDE-u} \gamma u_{tt}-\dive(\sigma) =  f && \text{in } \Omega
\times (0,T),
\\
\label{PDE-sigma-b}
& \sigma = a(x)e (u) && \text{in } \Omega_\eps \times (0,T), &&
\\
\label{PDE-sigma-s}
& \sigma = \lambda \mathrm{tr}(e(u))\mathbb{I} +2\mu e(u) + b
\xi  && \text{in } B_\eps \times (0,T), && \text{with } \xi \in \partial \calD(e(u_t)),
\\
\label{PDE-sigma-nojump}
&\JUMP{\sigma e_3}=0   && \text{on } S_\eps^\pm \times (0,T), &&
\\
\label{PDE-u-nojump}
&\JUMP{u}=0   && \text{on } S_\eps^\pm \times (0,T), &&
\\
\label{Dir}
&u=0 && \text{on } \Gdir \times (0,T), &&
\\
\label{Neu}
& \sigma n = g && \text{on } \Gneu \times (0,T), &&
\end{align}
\end{subequations}
where
$\mathbb{I}$ denotes the identity matrix and  $e(u)$ the symmetric linearized strain tensor related to the displacement vector $u$, defined by $e_{ij}(u)=\frac 12 (\partial_j u_{i}+\partial_i u_{j})$, $ i,j=1,\ldots,3$.
% and
%\EEE $f$ and $g$ are a given  body and surface force,
%respectively.
 We will supplement    system \eqref{PDE-original}  with the initial conditions
\begin{equation}
\label{Cauchy} u(0) = u_0, \qquad u_t(0)= v_0 \qquad \text{ in }
\Omega.
\end{equation}
 Note that \RRE the strong formulation \eqref{PDE-original} of the   problem in the case the
  thickness of the adhesive layer is strictly positive \EEE indeed corresponds to the formulation of the momentum balance equations written in the two bulk domains. However, in what follows we will be able to  provide an asymptotic result only for a variational (weak) formulation of system  \eqref{PDE-original} with the Cauchy conditions \eqref{Cauchy}, namely\EEE
  %\footnote{\LIC You are right, Christian: writing the strong formulation of the limiting problem would be a mistake, as it would make things unnecessarily complicated. However, we believe that it is important to write the strong formulation of the problem `with positive thickness', to make understanding easier for the reader.. you also wrote the strong formulation in \cite{LLOO}. Therefore, we suggest to keep \eqref{PDE-original}..}
% A  weak formulation for the Cauchy problem (\ref{PDE-original}--) reads
\\
\textbf{Problem (P):}
{\itshape Find $u: \Omega\times [0,T]\to \R^3$ sufficiently smooth fulfilling \eqref{Dir}, \eqref{Cauchy},
and such that there exists $\xi \in \partial \calD(e(u_t))$ satisfying}
\begin{equation}
\label{Ps}
\begin{aligned}
&
\int_\Omega \gamma u_{tt} \cdot v \dd x +\int_{\Omega_\eps} a e(u) \cdot e(v) \dd x +\int_{B_\eps} \rmD W_{\lambda\mu}(e(u)) \cdot e(v) \dd x
+b\int_{B_\eps} \xi \cdot e(v) \dd x
\\
&
=\int_\Omega f \cdot v \dd x +\int_{\Gneu} g \cdot v \dd \mathcal{H}^2(x)
\end{aligned}
\end{equation}
{\itshape for all $v$ sufficiently smooth in $\Omega$ and vanishing on $\Gdir$. }
\par
In the next section, we will show that Problem $(P)$ has a unique solution in a suitable sense.
In Sec.\ \ref{s:4} we will then
%\footnote{\LIC We have moved here (and slightly reformulated it) the text at the beginning of Section 4, on p.\ 6 of your handwritten notes. }
determine the asymptotic behavior of the solutions to Problem (P) when the quintuple of  geometrical and mechanical data  $(\eps,\lambda,\mu,b,\rho)$
that characterize the structure is regarded as a quintuple of  positive parameters $\nq n: = \qui n$,
%\footnote{\LIC Here we suggest a small change of notation w.r.t.\ \cite{LLOO}: the quintuple of parameters is denoted by $q$, in place of $s$; the running index  is $n$, instead  of the parameter $s$... For us, using $s$ is slightly confusing, because $s$ reminds us of the time variable. Instead, we are much more familiar with using $n$ as a the index of a sequence...
%anche qui, giustificare il cambio di notazione da $s$ a $q$... $s$ ricorda troppo i tempi..}
 suitably \EEE converging to a limit  $\nq \infty$
 (cf.\ the upcoming Hyp.\ \ref{hyp:params}). \EEE
 % with a suitable scaling.
%%%
%%%%%%
\section{Existence and uniqueness of solutions to Problem $(P)$}
\label{s:3}
We will rigorously prove our  existence result for  Problem $(P)$
 relying on the, by now classical, results from \cite{Brezis73}.
For the
  asymptotic analysis we shall
 resort \EEE to a nonlinear version of Trotter's theory of approximation of semigroups, acting on \emph{variable} Hilbert spaces. This approach in fact relies on a reformulation of system \eqref{Ps}  as an  abstract evolutionary equation  involving semigroups on suitable Hilbert spaces. In what follows, % upcoming Sec.\ \ref{ss:abstract}
 we recapitulate this formulation,  as proposed in \cite{LLOO}, and recall the existence result
  proved therein, cf.\ Theorem \ref{thm:LLOO} ahead.  Since we will keep the  quintuple of parameters $q= \quiname$ \emph{fixed} in this section, in the following lines
 we shall not highlight the dependence of the solution $u$, and of the functionals/operators/spaces entering into the variational formulation of Problem $(P)$, on $q$.
\par
From now on we will assume that the body and the surface forces fulfill
\begin{equation}
\label{ass-f&g}
f \in \BV([0,T];L^2(\Omega;\R^3)), \quad g \in  W^{2,\infty}(0,T; L^2(\Gneu;\R^3)).
\end{equation}
Here and in what follows, $\BV([0,T];X)$ shall denote the subspace   of $L^1(0,T;X)$ consisting of all the elements whose distributional derivative with respect to time is a bounded Radon measure on $(0,T)$,  valued in $X$.
  Along the footsteps of  \cite{LLOO} we seek a solution
  to Problem $(P)$
   % the solution was sought i
of the form
\begin{equation}
\label{decomp-u}
u= \ue + \ur.
\end{equation}
In \eqref{decomp-u}, $\ue$ is    the unique solution of the `stationary' problem
\begin{subequations}
\label{problem-ue}
\begin{equation}
\label{p-ue-1}
\ue(t) \in H_{\Gdir}^1 (\Omega;\R^3),  \qquad
\varphi(\ue(t),v) = L(t)(v) \quad \text{for all } v \in  H_{\Gdir}^1 (\Omega;\R^3) \text{ and all } t \in [0,T],
\end{equation}
where $ H_{\Gdir}^1 (\Omega;\R^3) $ denotes the closed subspace of $ H^1 (\Omega;\R^3) $
 consisting of the elements with zero  trace on $\Gdir;$
hereafter, the notation $H_{\Gamma}^{1}(G;\R^3)$ will be systematically used for any
$G\subset \R^3$, $\Gamma \subset \partial G $. % and for any $1\leq q \leq \infty$.
Furthermore,
\EEE
%\footnote{\LIC To simplify notation, we suggest to avoid introducing the functional $\Phi(v) =\varphi(v,v)$..}
\begin{align}
&
\label{p-ue-2}
\varphi(v,v') : = \int_{\Omega_\eps} a e(v) \cdot e(v') \dd x +
  \int_{B_\eps} \rmD W_{\lambda,\mu} (e(v)) \cdot e(v') \dd x  \quad \text{for all } v,\, v '  \in H_{\Gdir}^1 (\Omega;\R^3),
  \\
  &
\label{p-ue-3}
L(t)(v): = \int_{\Gneu} g(x,t) \cdot v(x) \dd \calH^2(x) \quad \text{for all } v \in  H_{\Gdir}^1 (\Omega;\R^3) \text{ and all } t \in [0,T].
  \end{align}
  \end{subequations}
%  For later use, we also introduce the functional  \footnote{\GGE se introduciamo $\Phi$ dovremmo anche introdurre $K$ maiuscolo, cosa che non facciamo. Allora potremmo semplicemente non introdurli...\EEE}
% \begin{equation}
% \label{funzPhi}
% \Phi: H_{\Gdir}^1 (\Omega;\R^3) \to [0,\infty) \qquad \Phi(v): = \varphi(v,v)\,.
% \end{equation}
 Since, for every fixed $t\in[0,T]$ the operator $g(t)\mapsto \ue(t)$ is linear and continuous from $L^2(\Gneu;\R^3)$
 to $H_{\Gdir}^1 (\Omega;\R^3)$, we find that
    %Classical arguments yield the existence of a unique solution to \eqref{problem-ue}
  \begin{equation}
  \label{ue}
  \ue \in W^{2,\infty}(0,T; H_{\Gdir}^1(\Omega;\R^3))\,.
  \end{equation}
  \par
   Given such $\ue$, the remaining part $\ur$  of $u$ can be  obtained by solving the following evolutionary problem:
   %\footnote{\LIC We have added the variational formulation of the problem for $\ur$ for consistency with \eqref{problem-ue} as well as   for the sake of  clarity.  We are convinced that the argument is easier to follow in this way....}
\\  {\itshape Find $ \ur: \Omega\times [0,T]\to \R^3$ sufficiently smooth fulfilling \eqref{Dir},  the Cauchy conditions
 $
  \ur(0) = u_0 -\ue(0)$ and  $\partial_t \ur(0)=   v_0- \partial_t \ue(0) $
   %$\footnote{\RRE Giovanna, hai ragionissima a chiedere che
%  $\partial_t \ur(0)=   v_0- \partial_t \ue(0) $. Pero' mi sembra che, nel suo lavoro su JMPA Licht chieda solo $\partial_t \ur(0)=   v_0$. Infatti, chiedere
 %  }
 in $\Omega$,
and such that there exists
 $\xi:B_\eps \to \R$ satisfying  $\xi \in \partial \calD(e(\partial_t (\ur+\ue)))$ a.e.\ in $B_\eps$ and }
  \begin{equation}
\label{variational-equation-url}
\begin{aligned}
&
\int_{\Omega} \gamma  \partial_{tt}(\ur (t){+}\ue(t)) \cdot v \dd x + \int_{\Omega_\eps} a e(\ur(t)) \cdot e(v) \dd x +
%\int_{B_\eps} \rho \partial_{tt} \ur(t) \zeta \dd x  +
  \int_{B_\eps} \rmD W_{\lambda,\mu} (e(\ur(t))) \cdot e(v) \dd x
+ b\int_{B_\eps} \xi  \cdot e(v) \dd x
\\
&
= \int_\Omega f(t) \cdot v \dd x
\qquad \text{for all } v \in H_{\Gdir}^1 (\Omega;\R^3).
\end{aligned}
\end{equation}
\par
In \cite{LLOO}, the existence of a (unique) solution to the Cauchy problem for \eqref{variational-equation-url} % Problem \ref{prob-ur}
was proved by reformulating it as an  abstract evolutionary  problem. Similarly arguing, we introduce the following problem
\begin{equation}
\label{abstract-problem}
\begin{cases}
\frac{\dd }{\dd t}\abv(t) + \abop \abv(t)  \ni  \abfo(t) \qquad \text{ in } \absp \ \  \foraa\, t \in (0,T),
\\
\abv(0) = \abv_{0},
\end{cases}
\end{equation}
with $\absp$ a (separable) Hilbert space (that will turn out to be the space of possible states with finite mechanical - i.e., (kinetic{+}strain) - energy),   $\abop : \absp  \rightrightarrows \absp$  a maximal monotone (multivalued) operator, and
 $\abfo \in L^1 (0,T;\absp)$.
More precisely, we consider  the space %the Hilbert space
\begin{subequations}
\label{Hilb}
\begin{equation}
\label{Hilbert}
\absp %=\absp_{(\eps,\lambda, \mu, \rho)} : =
: = H^1_{\Gdir}(\Omega;\R^3) \times L^2 (\Omega;\R^3)  %\quad \text{with elements } \abv = (u,v)
\end{equation}
endowed with the following inner product and induced norm
%(the dependence of $\absp$ on the parameters $(\eps,\lambda, \mu, \rho)$ is shown by  \eqref{inner-norm} below)
 %where for all $\abv = (u,v)$ and $\abv' = (u',v')$ in $\absp$,
%the  are
\begin{equation}
\label{inner-norm}
\begin{aligned}
&
(\abv,\abv') %= (\abv,\abv')_{(\eps,\lambda, \mu, \rho)}
: = \varphi(u,u') + k(v,v')\,, \quad  |\abv|^2=(\abv,\abv)
\qquad \text{for all }  \abv = (u,v),\, \abv' = (u',v')  \in \absp,
\\
& \qquad
    \quad \text{with } k(v,v') : = \int_{\Omega} \gamma  v v' \dd x  \text{ for all }  v,v'\in  L^2 (\Omega;\R^3)\,,
 \end{aligned}
 \end{equation}
 \end{subequations}
 and $\varphi$   defined in \eqref{p-ue-2}.
Observe that the norm induced   by the bilinear form $\varphi$ is equivalent to the standard Sobolev norm
on $H_{\Gdir}^1(\Omega;\R^3)$   by Korn's inequality.
%Observe that we use the notation $H_n(=\absp_{(\eps,\lambda, \mu, \rho)})$  for the ambient space in order to highlight the fact that the inner product from \eqref{inner-norm} does depend on $(\eps,\lambda, \mu, \rho)$.
We introduce the operator $\abop \colon \absp  \rightrightarrows   \absp$,
%\footnote{Ma $\abop$ non \`e un operatore multivoco? Non dovremmo scrivere $\abop: \absp \rightrightarrows \absp$??? \GGE si \EEE}
with domain
\begin{subequations}
\label{opA}
\begin{equation}
\label{domain-op}
\begin{aligned}
\rmD(\abop) \colon =   \Big\{   \abv = (u,v) \in \absp\, : \
&   (1) \  v \in  H_{\Gdir}^1(\Omega;\R^3)
\\
&
   (2)  \  \exists\, (w,\xi) \in L^2 (\Omega;\R^3)  \times  L^2(B_\eps;\R^3)  \EEE \text{ with } \xi \in   \partial \calD(e(v))   \text{ a.e.\ in } B_\eps,
 \text{ s.t. }
   \\
   & \qquad   k(w,v')   +  \varphi (u,v') + b \int_{B_\eps} \xi \cdot e(v') \dd x =0 \text{ for all } v' \in     H_{\Gdir}^1(\Omega;\R^3)
\Big\}\,,
\end{aligned}
\end{equation}
defined at every $\abv = (u,v)$ by
\begin{equation}
\label{ope}
\abop \abv : = \left(\begin{array}{cc}
-v
\\
0
\end{array} \right) +
 \left\{  \left(\begin{array}{cc}
0
\\
-w
\end{array} \right)  \, : \  \text{$w$ as in  \eqref{domain-op}(2)} \right\}. \EEE
\end{equation}
\end{subequations}
% Note that the operator $\abop$ is well defined (in particular $w$ in \eqref{domain-op} is uniquely determined by $\abv$). Then, we \EEE recall
 For later use, we recall the following result from \cite{LLOO}, establishing  a link between  the equation defining the resolvent of $\abop$ and the Euler-Lagrange equation for the
minimization of the functional
%\footnote{In \eqref{energy-eps}  e in \eqref{var-resolv}  dovremmo usare un'altra lettera, non $v$, per la variabile indipendente di $\calJ$.. $v$ fa pensare alle velocit\`a}
\begin{equation}
\label{energy-J-eps}
\begin{aligned}
\mathcal{J}:  H_{\Gdir}^1(\Omega;\R^3) \to \R, \quad  \mathcal{J}(v): = \frac12
 k(v,v) -  k(\psi_2,v) +\frac12 \varphi (v,v)   + \varphi(\psi_1,v) +
  b \int_{B_\eps} \calD(e(v)) \dd x,
 \end{aligned}
\end{equation}
with $(\psi_1,\psi_2) \in \absp $ given.
% (we choose not to highlight the dependence of $\mathcal{J}_{\param}$ on $(\psi_1,\psi_2)$ for the sake of simplicity).
\begin{proposition}[\cite{LLOO}, Prop.\ 3.1]
\label{prop:fromLLOO}
 The operator $\abop$ is  maximal monotone on $\absp$ and  its resolvent $(\mathsf{I}{+}\abop)^{-1}: \absp \to \absp$ is given
 \RRE for all $(\psi_1,\psi_2) \in  \absp$ by
$
\abv  =(\mathsf{I}{+}\abop)^{-1} (\psi_1, \psi_2)
$
%\qquad \text{ i.e. } \qquad
% \bar \abv + \abop \bar \abv  \ni  (\psi_1,\psi_2),
 %\]
if and only if
 \begin{equation}
   \label{def-resolv}
    \text{$\abv=  (\bar u,\bar v) $ with $\bar u$ and $\bar v$ fulfilling } \ \EEE
\begin{cases}
  \mathcal{J}(\bar v) \leq \mathcal{J}(v) \qquad \text{for all } v \in H_{\Gdir}^1(\Omega;\R^3),
 \\
 \bar u = \bar v +\psi_1\,.
 \end{cases}
 \end{equation}
\end{proposition}
\par
 We now consider
   the Cauchy problem \eqref{abstract-problem}  with $\absp$ from \eqref{Hilbert}, $\abop$ from \eqref{opA}, and the data $\abfo$ and
$\abv_0$ given by
\begin{equation}
\label{data-Cauchy-abs}
\abfo =  \left(-\partial_t \ue, \frac{f}{\gamma}\right), \qquad \abv_0 = (u_0,v_0) - (\ue(0), 0) \text{ such that } \abv_0 \in  \rmD(\abop).
\end{equation}
We denote by
 $\abv^{\mathrm{r}}  =  (\ur, v^{\mathrm{r}}) $ the  solution to \eqref{abstract-problem}, which exists, unique, in $ W^{1,\infty} (0,T;\absp)$ thanks to, e.g.,   \EEE
 \cite[Prop. 3.4]{Brezis73}.
By the careful definition of
$\absp$, $\abop$, and of the data $\abfo$ and $\abv_0$,
it can be easily checked that   $\ur$ \RRE and that   $v^{\mathrm{r}}=\partial_t (\ur+\ue)$ \EEE  solve
the Cauchy problem for \eqref{variational-equation-url}.
\par
\RRE Setting $u: = \ur + \ue$,
%taking tinto account the previously found
we ultimately find the unique solution to Problem (P). This is summarized in the following result. \EEE
 %operator such that for all {n}$ we have that $(\psi_1,\psi_2)\in \bar U+A\bar U$ (with $\bar U=) if and only if $\bar u=\bar v+\psi_1$ and $J(\bar v)\leq J(v)$ for all $v\in  \begin{equation}
% \label{def-resolv}$, where
%\begin{equation}
%{\mathcal J}_n(v)=\frac 1 2 K_{\eps,\rho}(v)-k_{\eps,\rho}(\psi_2,v) +\frac 1 2 \Phi_{\eps,\lambda,\mu}(v)  \varphi_{\eps,\lambda,\mu} (\psi_1,v) + b \int_{B_\eps} {\mathcal D}(e(v))\dd x
%\end{equation}
%with  $\Phi_{\eps,\lambda,\mu}(v)=\phi_{\eps,\lambda,\mu}(v,v)$ and $K_{\eps,\rho}(v)=k_{\eps,\rho}(v,v)$.
%\par
%\GGE Using \eqref{decomp-u}, \eqref{p-ue-1}, and \eqref{domain-op}, it \EEE can be checked that , supplemented with
%the data
%(observe that the function $\varrho$ fulfills
%$\varrho(x) \geq \min \{ \rho^*,\rho\}$ for almost all $x\in \Omega$),
%\GGE is equivalent to Problem (P) (cf. \eqref{Ps}). Actually, it
%yields a unique  \EEE solution $\abv^{\mathrm{r}} = (\ur, v^{\mathrm{r}})$ such that $\ur$  \EEE A result of \cite{????} therefore yields
\begin{theorem}[\cite{LLOO}, Thm.\ 3.1]
\label{thm:LLOO}
Let the data $f,\, g$ comply with \eqref{ass-f&g}  and
let $ (u_0,v_0)  \in (\ue(0),0) +  \rmD(\abop) $.
 Then,
  the Cauchy problem \eqref{abstract-problem} with $\absp$, $\abop$, and the data  $\abfo$ and $\abv_0$  from \eqref{Hilb}, \eqref{opA}, and \eqref{data-Cauchy-abs}, respectively,
 has a unique solution $\abv^{\mathrm{r}} = (\ur,\vr) \in W^{1,\infty} (0,T;\absp)$. Hence, there exists a unique $u\in W^{1,\infty} (0,T;H_{\Gdir}^1(\Omega;\R^3)) \cap W^{2,\infty} (0,T;L^2(\Omega;\R^3))$, given by   $u= \ue+\ur $, that satisfies   \eqref{Ps} for all $v\in H_{\Gdir}^1(\Omega;\R^3)$  and for almost all  $t\in (0,T)$  and  complies with \EEE  the Cauchy conditions \eqref{Cauchy}.
%  Problem
%\ref{prob:var-PDE} admits a unique solution $u \in  W^{1,\infty}(0,T;H_{\Gdir}^1(\Omega;\R^3)) \cap W^{2,\infty}(0,T;L^2(\Omega;\R^3)) $.\footnote{\RRE This is true if also $\ue $ belongs to $W^{2,\infty}(0,T; L^2(\Omega;\R^3)$, whereas for the moment we only have $u\in \BV^2([0,T]; H^1(\Omega;\R^3))$,  we have to pay attention to  this....}
\end{theorem}
\section{Asymptotic analysis}
\label{s:4}
In this section we address the asymptotic behavior of a sequence $(u_n)_n$ of solutions to Problems $(P_n)$ corresponding to a sequence $\nq n= \qui n$ of mechanical and geometrical parameters that satisfy the suitable  conditions.
As the overall density of the structure  depends on the parameter $\rho_n$, we shall denote it by $\gamma_n$.
%Such conditions
 The requirements in Hypothesis  \ref{hyp:params} below  in particular \EEE
 reflect  the
fact that the adhesive  layer has
 vanishing  thickness and that the total mass of the adhesive layer remains strictly positive.
 \begin{hypothesis}
\label{hyp:params}
We suppose that
\begin{equation}
\label{lim-quadruple}
\exists \, \lim_{n\to\infty} \nq n = \nq \infty = \qui \infty
\end{equation}
such that
\begin{enumerate}
\item $  \nq \infty \in \{0\} \times [0,\infty) \times [0,\infty) \times [0,\infty]\times \{\infty\}$;
\item $\exists\, (\bar\lambda,\bar\mu) \in [0,\infty]\times [0,\infty]$ s.t.\ $(\bar\lambda,\bar\mu) = \lim_{n\to\infty} \left( \frac{\lambda_n}{\eps_n}, \frac{\mu_n}{\eps_n}\right)$; % and satisfying \eqref{ellittico}; \PERME QS. FU AGGIUNTO DA ELENA, CREDO...\EEE
\item $\lim_{n\to\infty} b_n \eps_n =0$ and $\exists\, \bar b \in [0,\infty]$ s.t. $\bar b = \lim_{n\to\infty} \frac{b_n}{\eps_n^{p-1}}$,  where $p\in [1,2]$ is given as in \eqref{dissipation};
\item $\bar \mu \in (0,\infty]  $ if $\min \{ \mathcal{H}^2(\Gdir^\pm)\} =0$;
\item $\limsup_{n\to\infty} \frac{\eps_n^2}{\mu_n} <\infty$;
\item $\exists\, \bar\rho \in (0,\infty)$ s.t.\ $\bar\rho = \lim_{n\to\infty} \rho_n \eps_n$.
%\item $\exists\, \lim_{n\to\infty} \frac{\lambda_n}{2\eps_n} : = \bar{\lambda} \in [0,+\infty)$;
%\item $\exists\, \lim_{n\to\infty} \frac{\mu_n}{2 \eps_n} : = \bar{\mu} \in [0,+\infty)$;
%\item  $\exists\, \lim_{n\to\infty} \frac{b_n}{2 \eps_n} : = \bar{b} \in [0,+\infty]$;
%\item  $\exists\, \lim_{n\to\infty} \rho_n \eps_n : = \bar{\rho} \in (0,+\infty)$;
%\item $\limsup_{n\to\infty} \frac{\eps_n^2}{\mu_n} =0$.
\end{enumerate}
\end{hypothesis}
A comparison between Hypothesis \ref{hyp:params} and the analogous assumption \cite[(H2)]{LLOO}
reveals that the only difference lies in condition (6) on the asymptotic behavior of the sequence $(\rho_n)_n$;
 condition (6) does indeed encompass the fact that the total mass of the adhesive layer is strictly positive, in the limit.
 As we will see,
this  will make the limiting problem significantly different from that considered in \cite{LLOO}. \EEE % \PERME DA ESPANDERE.. \EEE
%\footnote{\LIC We have added here a small sentence comparing the two sets of conditions. Our aim was to highlight the role of condition (6) on the sequence $(\rho_n)_n$....} \EEE
\par
Actually, for the sake of simplicity and to highlight the main points in our analysis, in this  paper \EEE   we shall confine the discussion to the case in which  $(\bar\lambda,\bar\mu) \in  [0,\infty) \times (0,\infty)$. %\footnote{\PROB Christian, from the discussion of the `singular' cases that you have developed in Sec.\ 5, we would imagine that, in fact, we are confining the discussion to the case $(\bar\lambda,\bar\mu) \in [0,\infty) \times (0,\infty)$, correct?}
%and postpone the study of the other cases to Section \ref{s:5}.
In the upcoming Section \ref{ss:4.1}, with sort of  heuristic arguments  we propose a candidate PDE system for the description of the limiting behavior of the structure under Hypothesis
\ref{hyp:params} on the parameters $(q_n)_n$. As we will see,  such a  system may be somehow `guessed', also based on the analysis previously
performed  in \cite{LLOO},   cf.\ also \cite{Licht-Weller}.  \EEE

%\footnote{\PROB  In addition to
 %\cite{LLOO}, are there
% other papers by you, Christian, that should be cited?}
%\EEE
In accordance with the approach developed in Sec.\ \ref{s:3},  the limiting system \EEE will be formulated as an evolutionary equation in a suitable Hilbert space, governed by a nonlinear maximal  monotone operator.
Next, in Sec.\ \ref{ss:4.2} we will carry out the asymptotic analysis via (a version of) Trotter's theory of approximation of semigroups on variable Hilbert spaces.
\subsection{A candidate for the limiting behavior}
\label{ss:4.1}
The functional framework
for the limiting problem is naturally obtained by studying the behavior of sequences $(\abv_n)_n = (u_n,v_n)_n$ of possible mechanical states, uniformly bounded in the Hilbert spaces $\absp_n$ from
\eqref{Hilb}, namely
%\footnote{\LIC here we have recalled the form of the spaces $\absp_n$ and made  the bounds on the sequences $(u_n,v_n)_n$ more explicit, otherwise, we think, it would be very difficult to follow the arguments for convergences (1)--(7)...}
\begin{equation}
\label{Hilbert-spaces-n}
\begin{aligned}
&
\absp_n: = H^1_{\Gdir}(\Omega;\R^3) \times L^2 (\Omega;\R^3)  \text{ endowed with the norms $|\cdot|_n$  and the inner products }
\\
%\text{
% induced by the inner products
 &
 (\abv,\abv')_n %= (\abv,\abv')_{(\eps,\lambda, \mu, \rho)}
: = \varphi_n(u,u') + k_n(v,v')
\qquad \text{for all }  \abv = (u,v),\, \abv' = (u',v')  \in \absp \qquad  \text{with }
\\
&
\begin{cases}
\varphi_n(v,v') : = \int_{\Omega_\eps} a e(v) \cdot e(v') \dd x +
  \int_{B_{\eps_n}} \rmD W_{\lambda_n,\mu_n} (e(v)) \cdot e(v') \dd x  &  \text{ for all } v,\, v '  \in H_{\Gdir}^1 (\Omega;\R^3),
  \\
 k_n(v,v') : = \int_{\Omega_{\eps_n}} \rho^* v v' \dd x  + \int_{B_{\eps_n}} \rho_n v v' \dd x  &  \text{ for all }  v,v'\in  L^2 (\Omega;\R^3)\,.
 \end{cases}
 \end{aligned}
 \end{equation}
 \EEE
 %$(\cdot,\cdot)_n$ from \eqref{inner-norm} that depend on the data $(\lambda_n,\mu_n,\rho_n)_n$}
%(in what follows we will denote by $\varphi_n$ and $k_n$ the bilinear forms defined in  \eqref{p-ue-2} and \eqref{inner-norm}).
Therefore,  for uniformly bounded mechanical states $(\abv_n)_n =(u_n,v_n)_n$ the  following  estimates \EEE hold for a constant
$C>0$ uniform w.r.t.\ $n\in \N$:
%estimates
\begin{subequations}
\label{bounded-states}
\begin{align}
&
\label{bounded-states-1}
k_n(v_n,v_n) = \int_{\Omega_{\eps_n}} \rho^*(x) |v_n|^2 \dd x +  \int_{B_{\eps_n}} \rho_n |v_n|^2  \dd x \leq C,
\\
&
\label{bounded-states-2}
\varphi_n(u_n,u_n) = \int_{\Omega_{\eps_n}} a |e(u_n)|^2 \dd x +  \int_{B_{\eps_n}} \rmD W_{\lambda_n,\mu_n}(e(u_n)) \cdot e(u_n)  \dd x \leq C.
\end{align}
\end{subequations}
Let us now draw some conclusions from %estimates
 \eqref{bounded-states}.
\par
First of all, observe that, via a simple change of variables estimate \eqref{bounded-states-1} may be rephrased as
\begin{equation}
\label{bounded-states-1-altern}
k_n(v_n,v_n) = \int_{\Omega_{\eps_n}} \rho^*(x) |v_n|^2 \dd x + \rho_n \eps_n \int_{B}  |\sop {\eps_n}[v_n] |^2  \dd x \leq C
\end{equation}
in terms of the  operator
%\footnote{\PROB We should be careful with this operator: when we apply it in the following pages, we are often applying it to functions that are not just defined on $B_\eps$ (cf.\ e.g.\ when we write $\sop{\eps_n}[u_n]$ in the statement of Prop.\ \ref{prop4.3})..
%\\
%What we mean, in those cases,  is that we apply $\sop \eps$ to the restriction of a given function to  $B_\eps$, right?? In the same way, we have to carefully define its inverse...} \EEE
 $\sop \eps $ that maps a function $v$ into the function $\sop \eps [v]$ defined by \EEE
\begin{equation}
\label{sop-eps}
\begin{aligned}
\sop \eps [v](\hat x,x_3) : = v(\hat x,\eps x_3) \quad &  \text{for all } x =(\hat x, x_3) \in B: = S \times (-1,1)
\\
&
\text{and all measurable functions $v$ on $B_\eps = S \times (-\eps,\eps)$.}
\end{aligned}
\end{equation}
 By virtue of Hypothesis \ref{hyp:params}(6)  and
condition
\eqref{rho-star} on $\rho^*$,    from the bound \eqref{bounded-states-1-altern} % the bound
%\[
%^\sup_{n\in \N} k_n(v_n,v_n) \leq C
%\]
%along the sequence $(v_n)_n$,
  we deduce that the pair
  \begin{enumerate}
  \item
  $( \chi_{\Omega_{\eps_n} }  v_n, \sop{\eps_n}[v_n])_n$, up to a subsequence,
weakly converge,
\begin{equation}
\label{vo-vb}
\exists\,
(\vo v,\vb v)
\in L^2(\Omega;\R^3) \times L^2(B;\R^3)\, : \qquad  \chi_{\Omega_{\eps_n} }  v_n\weakto \vo v \text{ in } L^2(\Omega;\R^3), \  \sop{\eps_n}[v_n]\weakto \vb v \text{ in }  L^2(B;\R^3).
\end{equation}
\EEE
\end{enumerate}
% (we recall that $\chi_O$ denotes the characteristic function of the measurable set $O$)
 Thus, we may describe the limiting kinetic state by \emph{two velocity fields} $\vo v$ and $\vb v$
that effectively represent the limiting behavior of the velocity in the adhering bodies and in the adhesive layer, respectively.
\par
Secondly, we may deduce that
there exists a pair $(\vo u,\vb u) \in H_{\Gdir}^1 (\Omega{\setminus}S;\R^3) \times L^2(B;\R^3)$ such that, up to a subsequence,
\begin{enumerate}
 \setcounter{enumi}{1}
\item  the functions $\chi_{\Omega_{\eps_n}} e(u_n)$ %\footnote{il suo $1_{\Omega_\eps}$ \`e una fz.\ caratteristica, vero?  \GGE si \EEE}
converge weakly to $e(\vo u)$ in
 $ L^2(\Omega;\R^{3 \times 3})$;
\item the traces on $S_{\eps_n}^\pm$ of $u_n$, considered as elements of $L^2(S;\R^3)$, converge  to the traces
on $S$
of $\vo{u}^\pm$ (i.e., the restrictions of $\vo u$ to $\Omega^\pm$) strongly in $L^2(S;\R^3)$;
\item the functions $\sop{\eps_n} [u_n]$ converge to $\vb u$ weakly in $L^2(B;\R^3)$.
\end{enumerate}
Let us shortly justify properties (2)--(4).
%\footnote{\LIC we have introduced some explanations because, it seems to us, deducing  (2)--(4) from  \cite[Prop.\ 4.3]{LLOO} and  \cite[Lemma 4.2]{LLOO} needs a bit of effort. The remarks we have introduced might guide the reader }
Indeed, following the lines of the proof of
  \cite[Lemma 4.2]{LLOO}, from the first two  bounds in   \eqref{bounded-states-1} and  \eqref{bounded-states-2}, via Korn's inequality and a standard diagonalization argument we infer that there exists $\vo u \in \bigcup_{\eta>0} H_{\Gdir}^1 (\Omega_\eta;\R^3)$  (with $\Omega_\eta = \Omega {\setminus} \overline{B}_\eta$) such that, up to a subsequence,  for all $\eta>0$ there holds  $u_n\weakto \vo u $ \EEE weakly in $ H^1 (\Omega_\eta;\R^3)$. In turn, there exists $e^* \in L^2(\Omega;\Sym)$ such that  $ \chi_{\Omega_{\eps_n}} e(u_n) \weakto e^*$ \EEE weakly in $ L^2(\Omega;\Sym)$. Clearly, the restriction of $e^*$ to any $\Omega_\eta $ coincides with the restriction to $\Omega_\eta$ of
  $e(\vo u) $ (i.e.\  the
   symmetric  part of the distributional gradient of $\vo u$). We thus conclude that $e(\vo u) \in L^2(\Omega;\Sym)$ (whence $\vo u \in H_{\Gdir}^1 (\Omega{\setminus}S;\R^3)$), and convergence (2) holds.
     Next, we repeat the very same arguments as in Step $2$ of the proof of  \cite[Prop.\ 4.3]{LLOO}, to conclude convergence (3).
   In order to check (4), we will first of all show that $ (\sop{\eps_n} [u_n])_n$ is bounded in  $L^2(B;\R^3)$. To this end, \EEE
    we first of all employ the key inequality
   %\footnote{\PROB Christian,  there must be a factor $\eps_n$ in front of  $\int_{S_{\eps_n}} |w|^2 \dd \hat x $, right? }
\begin{equation}
\label{e:4.1}
\begin{aligned}
&
\exists\, C>0 \ \forall\, n \in \N \  \forall\, w \in H_{\Gdir}^1(\Omega;\R^3)\, :
\\
&
\frac12 \int_{B_{\eps_n}} |w|^2
\dd x
%\GGE \dd \hat x \EEE
\leq  \eps_n  \int_{S_{\eps_n}} |w|^2  \dd \hat x   + C  \eps_n^2 \EEE  \int_{B_{\eps_n}} |e(w)|^2 \dd x
+C\eps_n^2 \int_{\Omega_{\eps_n}}  |e(w)|^2 \dd x,
\end{aligned}
\end{equation}
(cf.\
\cite[(4.20)]{LLOO}),
 whence
\[
\frac12 \int_{B_{\eps_n}} |u_n|^2
\dd x \leq \eps_n \int_{S_{\eps_n}} |u_n|^2 \dd x +C\eps_n^2 \left( \int_{\Omega_{\eps_n}}  |e(u_n)|^2 \dd x{+}   \int_{B_{\eps_n}} |e(u_n)|^2 \dd x\right) \,.
\]
Now, since
\[
\int_{S_{\eps_n}} |u_n|^2 \dd x \to  \int_{S} \left( |\gamma_S (\vo{u}^+)|^2{+} |\gamma_S (\vo{u}^-)|^2\right) \dd x,
\]
we find that the first integral on the right-hand side is estimated by $C\eps_n$.
Furthermore, by \eqref{bounded-states-2} we have that
\[
\int_{\Omega_{\eps_n}}  |e(u_n)|^2 \dd x  \leq C.
\]
Finally,
\[
\eps_n^2  \int_{B_{\eps_n}} |e(u_n)|^2 \dd x \stackrel{(1)}{\leq} \frac{\eps_n^2}{\mu_n}  \int_{B_{\eps_n}} \rmD W_{\lambda_n,\mu_n} (e(u_n)) \cdot e(u_n) \dd x
  \stackrel{(2)}{\leq} C\eps_n,
\]
where (1) follows from \eqref{ellittico} and (2) from the assumption that $\tfrac{\mu_n}{\eps_n} \to \bar\mu>0$, cf.\ Hypothesis \ref{hyp:params}
  (recall that we confine here our analysis to the case $\bar\mu\in (0, +\infty)$).
All in all, we conclude that
\[
\eps_n \int_{B} |\sop{\eps_n}[u_n]|^2 \dd x  = \int_{B_{\eps_n}}|u_n|^2 \dd x \leq C \eps_n \qquad \text{whence} \qquad \int_{B} |\sop{\eps_n}[u_n]|^2 \dd x\leq C.
\]
and convergence (4) follows.  \EEE
%
%
%
% we  use  Hypothesis \ref{hyp:params} in order to deduce that
%\begin{equation}
%\label{sop-n}
% \int_{B_{\eps_n}} |u_n|^2
%\dd x   \leq C\,.
%\end{equation}
%
%%\footnote{\PROB Christian, is the above argument correct? In fact, we were not sure of the role of estimate \eqref{e:4.1}... it is used to ensure \eqref{sop-n}, which in turn enters in the proof of convergence (3), right?} \EEE
%Convergence (4) is a trivial consequence of the fact that $ (\sop{\eps_n} [u_n])_n$ is bounded in  $L^2(B;\R^3)$,  due to \eqref{sop-n} and an obvious change of variables.
\par
Moreover,
since
\begin{equation}
\label{newlabel}
  \int_{B_{\eps_n}} |e(u_n)|^2 \dd x =\frac{1}{\eps_n} \int_{B} |e(\eps_n, \sop{\eps_n} [u_n])|^2 \dd x, \EEE
\end{equation}
where we have
introduced the  notation
%tensor
% used the notation
\[
e(\eps, w)_{i,j}: = \begin{cases}
\eps e(w)_{i,j} & \text{for } 1\leq i,j\leq 2
\\
\frac12 ( \eps \partial_i w_3+ \partial_3 w_i) & \text{for } 1\leq i \leq 2, \ j=3
\\
\partial_3 w_3 & \text{for } i=j=3
\end{cases}
\qquad \text{ for all } w \in H^1(B;\R^3),
\]
 the   convergence in the sense of distributions of $\sop{\eps_n}[u_n]$ implies that
\begin{enumerate}
 \setcounter{enumi}{4}
\item $e(\eps_n , \sop{\eps_n}[u_n])$ converge  to $\partial_3 \vb u  \symprod e_3$  weakly in $L^2(B;\R^3\times \R^3)$;
\item $\vb u  \in H_{\partial_3} (B;\R^3)$,
 with
\begin{equation}
\label{space-H-partial3}
 H_{\partial_3} (B;\R^3):= \{ u \in L^2(B;\R^3)\, : \  \partial_3 u  \in L^2(B;\R^3)\};
 \end{equation}
\item the traces of $\vb u $ on  % $S^+ : =  S \times \{1\}$ and $S^- : =  S \times \{-1\}$,
\begin{equation}
\label{s-pm}
  S^+ : =  S \times \{1\} \text{  and  } S^- : =  S \times \{-1\},
\EEE
%\quad \text{ and by }  \gamma_{S^\pm} \text{ the traces  on } S^\pm\,.
\end{equation}
% \pm e_3$,\footnote{ma che cosa intende con $S \pm e_3$
%\GGE sono i bordi superiore e inferiore (di quota $x_3=1$ e $x_3=-1$) di $B$ \EEE}
 hereafter denoted by $\gamma_{S^\pm}(\vb u) $ and treated as elements of $L^2(S;\R^3)$, coincide with  the traces on $S$ of $\vo{u}^\pm$, denoted by $\gamma_S(\vo{u}^\pm)$.
\end{enumerate}
Indeed, %\footnote{\LIC Again, a small explanation was in order, in our opinion..}
while  items (5) \& (6) are obvious, (7) follows from observing that the traces of $u_n$ on $S_{\eps_n}^\pm$ coincide with the traces of
$(\sop{\eps_n} [u_n])^\pm $ on $S^\pm$, and then taking the limit as $n\to\infty$.
\par
In view of the above considerations, we thus expect that the Hilbert space of
possible
 limiting states with finite energy will be
 \begin{subequations}
 \label{limiting-space}
 \begin{align}
 &
 \label{limiting-space-1}
 \absp: =  \Spu\times \Spv \text{ with }
 \\
 &
  \label{limiting-space-2}
 \Spu : = \{ u =(\vo u,\vb u) \in H^1(\Omega{\setminus} S;\R^3) \times H_{\partial_3} (B;\R^3) \, : \  \gamma_S ({\vo u}^\pm) = \gamma_{S^\pm} (\vb u) \},
 \\
&
 \label{limiting-space-3}
 \Spv : = \{ v =(\vo v,\vb v) \in L^2(\Omega;\R^3) \times L^2 (B;\R^3)  \},
 \end{align}
 \end{subequations}
endowed with the inner product (and related norm)
\begin{subequations}
\label{limit-inner-prod}
\begin{align}
\label{ip1}
&
(\abv,\abv') = \varphi(u,u') + k(v,v'), \quad |\abv|^2: = \varphi(u,u) + k(v,v) \quad  \text{for all $\abv =(u,v), \, \abv' = (u',v') \in \absp$},
\intertext{where}
&
\label{ip2}
 \varphi(u,u')  : = \int_{\Omega{\setminus}S} a e(\vo u) \cdot e(\vo{u}') \dd x + \int_B \rmD W_{\bar\lambda,\bar\mu}(\partial_3 \vb {u} \symprod e_3) \cdot (\partial_3 \vb {u}' \symprod e_3)  \dd x,
\\
&
\label{ip3}
k(v,v') = \int_\Omega \rho^*   |\vo v|^2 \EEE \dd x + \bar \rho \int_B |\vb v|^2 \dd x \,.
\end{align}
 \end{subequations}
 The limiting   pseudopotential of dissipation is   defined  by % \PERME MOTIVARE QUESTA FORMULA... \eee
 %\footnote{\PROB Christian,
%\eqref{dissipation-limit-c} is an additional assumption, that will go into the statement of our convergence result together with Hypothesis \ref{hyp:params}, right? Therefore, we have given it the status of a hypothesis.
% } \EEE
 %\footnote{\RRE assumed??  non mi \`e chiaro perch\'e la \eqref{dissipation-limit-c} venga assunta.. \GGE boh...\EEE}
 \begin{subequations}
 \label{dissipation-limit}
 \begin{align}
 &
  \label{dissipation-limit-a}
  \overline{\calD}: L^2(B;\R^3) \to [0,\infty] \quad \qquad
 \overline{\calD}(q) = \begin{cases}
 \bar b \mathcal{D}^{\infty, p} (q\symprod e_3) & \text{if } \bar b \in [0,\infty),
 \\
 I_{\{0\}} (q \symprod e_3) & \text{if } \bar b =\infty,
 \end{cases}
 \intertext{where $ I_{\{0\}} $ is the indicator function of $\{  0 \}$ and}
 &
  \label{dissipation-limit-b}
 \calD^{\infty,p} (e') = \limsup_{t\to\infty}   \frac{\calD(te') } {t^p}  \quad \text{for all } e' \in \Sym
 %\intertext{where it is assumed that}
 \end{align}
  \end{subequations}
   and $p\in [1,2]$ is given as in \eqref{dissipation}.  It is not difficult to check that
  \begin{equation}
  \label{null-indicator}
   I_{\{0\}} (q \symprod e_3) <\infty \ \Leftrightarrow \ q=0\,.
  \end{equation} \EEE
  In what follows, we shall assume
   \begin{hypothesis}
\label{hyp:dissipation}
We suppose that
  \begin{equation}
    \label{dissipation-limit-c}
 \exists\, \delta> 0 \  \exists\, \theta \in (0,p) \ \forall\, e' \in \Sym \, : \quad |\calD(e') {-} \calD^{\infty,p}(e')| \leq \delta (1{+} |e'|^\theta).
 \end{equation}
 \end{hypothesis}
 \par
  Hence, we can introduce  the evolution operator $\abop : \absp  \rightrightarrows  \absp$  with domain (cf.\ \eqref{null-indicator}) \EEE
  %\footnote{Non ci vorr\`a un qualche $b$ davanti al termine $ \int_{B} \xi \cdot (\partial_3 \vb{v'} {\symprod} e_3)$?? \GGE No, perch\`e $\bar{b}$ \`e gi\`a incluso nella definizione di $\overline{\calD}$ \EEE }
 \begin{subequations}
\label{opA-lim}
\begin{equation}
\label{domain-op-lim}
\begin{aligned}
\rmD(\abop) \colon =   \Big\{   \abv = (u,v) \in \absp\, : \
&   (1) \, v \in \Spu \text{ and } \partial_3 \vb u =0 \text{ if } \bar b = \infty,
\\
&
   (2)  \, \exists\, (w,\xi) \in \Spv  \times  L^2(B;\R^3)  \text{ s.t. }  \xi \in   \partial \overline \calD (\partial_3 \vb v \symprod e_3)   \text{ a.e.\ in } B,  \EEE
   \\
   & \qquad   k(w,v')   +  \varphi (u,v')   +  \int_{B} \xi \cdot (\partial_3 \vb{v}' {\symprod} e_3) \dd x =0 \text{ for all } v' \in   \Spu \text{ as in (1) } \Big\}\,,
\end{aligned}
\end{equation}
defined at every $\abv = (u,v)$ by
\begin{equation}
\label{ope-lim}
\abop \abv : = \left(\begin{array}{cc}
-v
\\
0
\end{array} \right) +  \left\{
 \left(\begin{array}{cc}
0
\\
-w
\end{array} \right)  \, : \  \text{$w$ as in  \eqref{domain-op-lim}(2)} \right\}. \EEE
\end{equation}
\end{subequations}
% Note that the operator $\abop$ is well defined (in particular $w$ in \eqref{domain-op} is uniquely determined by $\abv$). Then, we \EEE recall
With the same arguments  as for the operator $\abop$
from \eqref{opA},
 %associated with the parameters $\qui n$,
 it can be easily proved that $\abop$ is maximal monotone in $\absp$. Moreover, %\RRE we point out for later use  \EEE
   %for all $\Psi = (\psi_1,\psi_2) \in \absp$ one has
   with the very same arguments as for  Proposition \ref{prop:fromLLOO},
   one can check that the resolvent of $\abop$
 satisfies
 %\footnote{$\abop$ non \`e multivoco? Se \`e cosi', dovremmo scrivere $ \bar \abv + \abop \bar \abv \ni (\psi_1,\psi_2)$ }
\[
 \bar \abv + \abop \bar \abv  \ni  (\psi_1,\psi_2) \qquad \text{for every } (\psi_1,\psi_2)\in  \absp,
\]
 with $\bar \abv=(\bar u,\bar v)$,
 if and only if
 \begin{subequations}
 \begin{equation}
  \label{def-resolv-lim}
\begin{cases}
 \mathcal{J}(\bar v) \leq \mathcal{J}(v) \qquad \text{for all } v \in  \Spu,
 \\
 \bar u = \bar v +\psi_1\,
 \end{cases}
 \end{equation}
 with
 %\footnote{In \eqref{energy-eps}  e in \eqref{var-resolv}  dovremmo usare un'altra lettera, non $v$, per la variabile indipendente di $\calJ$.. $v$ fa pensare alle velocit\`a. Inoltre, non ci vuole qualche $b$ davanti al termine $\int_B \overline{\calD} (\partial_3 \vb{v} {\symprod} e_3) \dd x$?}
 \begin{equation}
\label{energy-J-eps-new}
\begin{aligned}
\mathcal{J}:  \Spu \to \R, \quad  \mathcal{J}(v): = \frac12
 k(v,v) -  k(\psi_2,v) +\frac12 \varphi (v,v)   + \varphi(\psi_1,v) + \int_B \overline{\calD} (\partial_3 \vb{v} {\symprod} e_3) \dd x \,.
 \end{aligned}
\end{equation}
 \end{subequations}
\par
 By arguing as in Section 3, the expected limit of the sequence  $(u_n : = \uen + \urn)_n$  (cf.\ \eqref{decomp-u})
of the solutions to Problems $(P_n)$
 will be the sum of some $\ue$, solution to a  limiting `stationary' problem and some $\ur$, solution to a limiting evolutionary problem. More precisely,
  we will have that \EEE
$\ue =(\voe{u}, \vbe{u}) \in W^{2,\infty} (0,T;\Spu)$, with $\vbe{u}$ affine in $x_3$    (cf.\  Remark \ref{rmk:affine}
ahead)
%\footnote{\LIC Christian, since it is not obvious to us why $\vbe{u}$  is affine in $x_3$, we have added a remark about this. Could you please check if it is correct?}
%\footnote{ a me non \`e cosi' ovvio... perch\`e $\vb {\ue}$ dovrebbe essere affine? \GGE l'ho recuperato da dei suoi appunti... comunque non \`e ovvio...forse \`e il caso di precisare com'\`e fatta $\vb{\ue}$ in una remark... \RRE io per il momento non l'ho ancora fatto, lo metto alla fine..\EEE }
is the unique solution to
\begin{equation}
\label{stati-ue-lim}
\varphi(\ue(t),v) = L(t)(v) \quad \text{for all } v \in \Spu \quad \text{for all } t \in [0,T],
\end{equation}
with $\varphi$ from \eqref{ip2} and $L$ from \eqref{p-ue-3}.  Instead, $\ur$ is the first component of the  solution
$\abv^{\mathrm{r}} = (\ur,\vr)$
 to the following abstract problem
%we are now in a position to introduce an evolution equation in $\absp$ which will be shown to describe the asymptotic behavior of $(u_n)_n$
%\footnote{copio, ma qui \`e scritto male. Dobbiamo introdurre prima lo splitting $u = \ue + \ur$, come abbiamo fatto nella Sez.\ 3, poi dare il problema risolto da $\ue$, e poi dire che $\ur$ risolve la \eqref{lim-abseq}... \GGE Infatti. Qui l'ho un po' riscritto....ma mi piacerebbe aggiungere la formulazione debole del pb limite in modo che si capisca un po' di pi\`u cosa risolve $\ur$ che sar\`a anch'esso del tipo $\ur= (\ur_{\Omega}, \ur_B)$ \EEE \RRE per il momento su questo soprassiedo..}
\begin{equation}
\label{lim-abseq}
\begin{cases}
\frac{\dd }{\dd t}\abv(t) + \abop \abv(t)   \ni \EEE \abfo(t) \qquad \text{ in } \absp \ \  \foraa\, t \in (0,T),
\\
\abv(0) = \abv_{0},
\end{cases}
\end{equation}
with $\abv_0$ specified later on and $\abfo$ given by
\begin{equation}
\label{417}
\abfo = \left( {-} \partial_t \ue, \left( \frac{f}{\rho^*},0 \right) \right)\,.
\end{equation}
We postpone to Section \ref{s:5} some comments on the variational formulation of the  initial-boundary value problem  \eqref{lim-abseq}, cf.\ \eqref{5.1} and \eqref{5.1bis} ahead.
\par
In the same way as for the Cauchy problem \eqref{abstract-problem} with the forcing term $\abfo$ from \eqref{data-Cauchy-abs},  the results in  \cite[Sec.\ III.2]{Brezis73})
yield \EEE
\begin{proposition}
If $\abv_0 \in \mathrm{D}(\abop) $ and if $(f,g)$ satisfy \eqref{ass-f&g}, then \eqref{lim-abseq} has a unique solution
$\abv^{\mathrm{r}} = (\ur,\vr) \in W^{1,\infty} (0,T;\absp)$. %\RRR and the first line of ... holds almost everywhere??? NON CAPISCO PI\`U \EEE
% Hence, there exists a unique $u\in W^{1,\infty} (0,T;H_{\Gdir}^1(\Omega;\R^3)) \cap W^{2,\infty} (0,T;L^2(\Omega;\R^3))$, given by   $u= \ue+\ur $, which satisfies   \eqref{Ps} for all $v\in H_{\Gdir}^1(\Omega;\R^3)$ and  the Cauchy conditions \eqref{Cauchy}.
\end{proposition}
We set
\begin{equation}
\label{added-label-Gio}
\abv^{\mathrm{e}} =(\ue, 0), \qquad \abv = \abv^{\mathrm{r}} + \abv^{\mathrm{e}}
\end{equation}
\par
We are now in a position to outline %\footnote{\LIC Christian, here we have given a scheme of the argument.. just for the sake of clarity}
our argument for proving the convergence of the sequence $(u_n = \uen + \urn)_n$ to $u = \ue +\ur$:
\begin{enumerate}
\item the convergence of $\uen $ to $\ue$ will be obtained in Proposition \ref{prop:4.5} ahead, as part of the proof of
\item the convergence of $\urn $ to $\ur$, stated in Theorem \ref{thm:main} ahead.
 For proving it,  we will resort to a nonlinear version of Trotter's theory of approximation of semigroups acting on  \emph{variable} spaces, as developed in the Appendix of \cite{ILM09}. The need for such a theory is motivated by the fact that the functions $\urn$ and $\ur$ do not  belong to \EEE the same space.
\end{enumerate}
The proof of  Theorem \ref{thm:main} will be carried out throughout Section \ref{ss:4.2}.
\RRE
\par We now conclude this section by specifying
 the structure of the solution
$\ue =(\voe{u}, \vbe{u})$ to the limit stationary problem
\eqref{stati-ue-lim}. In particular, we will show that $\vbe{u}$ is affine in $x_3$. % cf.\ \eqref{ub affine} ahead. \EEE
 %\RRR QUI NON CAPISCO PIU' NULLA.. \EEE
 \begin{remark}
\upshape
\label{rmk:affine}
Let
$\ue =(\voe{u}, \vbe{u}) \in W^{2,\infty} (0,T;\Spu)$ be the unique solution to
\eqref{stati-ue-lim}
with $\varphi$ from \eqref{ip2} and $L$ from \eqref{p-ue-3}. Then $\ue$ satisfies
\begin{equation}
\label{ue limite}
\begin{aligned}
&
\int_{\Omega{\setminus}S} a e(\voe u)(t) \cdot e(\vo{v}) \dd x + \int_B \rmD W_{\bar\lambda,\bar\mu}(\partial_3 \vbe {u} (t) \symprod e_3) \cdot (\partial_3 \vb {v} \symprod e_3)  \dd x
\\
&
=\int_{\Gneu} g(t) \cdot \vo{v} \dd \mathcal{H}^2(x)
\qquad \text{for all } v =(\vo{v}, \vb{v}) \in \Spu
\text{ and for all } t \in [0,T].
\end{aligned}
\end{equation}
 Choosing now $\vo v =0$ and $\vb v $ as an arbitrary test function $\varphi \in \mathrm{C}^\infty_{\mathrm{c}}(B;\R^3)$ leads to
 \[
 \int_{B} \left(\bar\mu\, \partial_3 ( \vbe{u})_1 \,\partial_3 \varphi_1 {+} \bar\mu \,\partial_3 ( \vbe{u})_2\, \partial_3 \varphi_2{+}(\bar\lambda {+}2\bar\mu )\,\partial_3 ( \vbe{u})_3 \, \partial_3 \varphi_3\right) \dd x =0 \qquad \text{for all } \varphi \in  \mathrm{C}^\infty_{\mathrm{c}}(B;\R^3),
 \]
  where $\partial_3 (\vbe{u})_i$ denotes the $i$th-component of the vector $\partial_3 \vbe{u}$. This
 implies that the function  $x_3 \mapsto \vbe{u}(\hat{x}, x_3,t)$ is affine.
 %  namely \PERME COERENTE CON QUELLO CHE STA SOPRA??
 %\begin{equation}
%\label{ub affine}
%\begin{aligned}
%\vbe{u}(\hat{x}, x_3,t):= \frac 12 \left(\gamma_S(u^e_{\Omega^+})+\gamma_S(u^e_{\Omega^-})\right)(\hat{x}, 0,t)+
%\frac 12 x_3 \left(\gamma_S(u^e_{\Omega^+})-\gamma_S(u^e_{\Omega^-})\right)(\hat{x}, 0,t).
%\end{aligned}
%\end{equation}
  \EEE
%Choosing $(0, \vb{v}) \in \Spu$ as admissible test function in \eqref{ue limite}, we verify that
%$\vbe{u}$ is linked to the traces of $\voe{u}$  on $S$  by the following relation
%\begin{equation}
%\label{ub affine}
%\begin{aligned}
%\vbe{u}(\hat{x}, x_3,t):= \frac 12 \left(\gamma_S(u^e_{\Omega^+})+\gamma_S(u^e_{\Omega^-})\right)(\hat{x}, 0,t)+
%\frac 12 x_3 \left(\gamma_S(u^e_{\Omega^+})-\gamma_S(u^e_{\Omega^-})\right)(\hat{x}, 0,t)
%\end{aligned}
%\end{equation}
% for all $(\hat x, x_3) \in B$
%and for all  $t \in [0,T]$, hence $x_3 \mapsto \vbe{u}(\hat{x}, x_3,t)$ is affine. To check \eqref{ub affine}, it is sufficient to
%observe that
% $\gamma_S (u^e_{\Omega^\pm})=\gamma_{S^\pm}(\vbe{u})$ and moreover that
%\[
%\begin{aligned}
%&
%\int_B \rmD W_{\bar\lambda,\bar\mu}(\partial_3 \vbe {u} (t) \symprod e_3) \cdot (\partial_3 \vb {v} \symprod e_3)  \dd x
%\\
%&
%=\int_S\rmD W_{\bar\lambda,\bar\mu}(\partial_3 \vbe {u} (t) \symprod e_3) \dd \hat x \cdot \int_{-1}^{1}(\partial_3 \vb {v} \symprod e_3)  \dd x_3 = 0
%\qquad
%\text{ for all } t \in [0,T]
%\end{aligned}
%\]
%due to the fact that $(0, \vb{v}) \in \Spu$ and hence $\gamma_{S^\pm}(\vb{v})=0$.
\end{remark}

\subsection{Convergence}
\label{ss:4.2}
Throughout this section,  we will implicitly assume the validity of Hypotheses \ref{hyp:params} and \ref{hyp:dissipation}, and of conditions \eqref{ass-f&g} on the problem data $f$ and $g$. In particular, we shall omit to invoke these assumptions in all  of the statements of the various results, with the exception of Theorem \ref{thm:main}.
\par
%As already mentioned, to prove the convergence of the sequence $(u_n)_n$ to $u = \ue +\ur$,
In the next subsection
%\footnote{\LIC Christian, it seems to us that we should definitely recall the basics of the result from Trotter theory here, before
%definiting explicitly the operators $\projn n$. Otherwise it is really difficult to follow the argument developed from Sec.\ \ref{sss:4.1.2} on. Therefore, we suggest to recall the result from
 %\cite{ILM09} right now, in Sec.\ \ref{sss:4.1.1}, and then to establish the setup for Trotter convergence in Sec.\ \ref{sss:4.1.2}.}
 we shortly recapitulate the basics of the result from Trotter theory that we shall use to prove
 Theorem \ref{thm:main}.
\subsubsection{Recaps on Trotter's theory of approximation of semigroups}
\label{sss:4.1.1}
Let us first fix some preliminary definitions.
 We consider
 \[
 \text{a sequence }  (\abspn{n})_n  \text{ of Hilbert spaces, with inner products } (\cdot,\cdot)_n \text{ and norms } |\cdot|_n,
 \]
 and a `limiting' Hilbert space $\absp$, such that
 for every $n \in \N$ there is defined an operator
 $\projn n :  \absp\to \abspn{n}$, linear and continuous,
 fulfilling the following properties:
 \begin{subequations}
 \label{props-Pn}
 \begin{align}
 &
 \label{prop-1-to-ver}
 \text{There exists $C>0$ such that for every $n\in \N$ and  $\abv\in \absp$ there holds $|\projn{n} \abv|_n \leq C |\abv|$;}
 \\
 &
  \label{prop-2-to-ver}
 \text{
 For every $\abv\in \absp$ there holds $\lim_{n\to\infty}|\projn{n} \abv|_n = |\abv|$.}
 \end{align}
 \end{subequations}
 Next,  for  a given sequence $(\abv_n)_n$ with $\abv_n \in \absp_n$ for every $n\in \N$, we will say that
\begin{equation}
\label{4.22}
\text{$(\abv_n)_n$  converge to $\abv \in \absp$  in the sense of Trotter if } \quad \lim_{n\to\infty} |\projn n \abv {-} \abv_n|_n=0.
\end{equation}
 We are now in a  position to recall the result from  \cite{ILM09} needed for our analysis.
 \begin{theorem}[\cite{ILM09}, Thm.\ 5]
 \label{thm:Trotter}
 Suppose that  the Hilbert spaces $\abspn {n},\, \absp $ fulfill \eqref{props-Pn}.
 Let $ \abopn{n}  :  \abspn{n} \rightrightarrows  \abspn{n} $, $ \abop : \absp \rightrightarrows \absp $ be maximal monotone operators,
 let $\abfon{n} \in L^1(0,T; \abspn n)$, $\abfo \in L^1(0,T;\absp)$, and let $\abv^0_n \in \overline{\rmD(\abopn n )}$, $\abv^0 \in \overline{\rmD(\abop)}$.
 Let $(\abv_n)_n$, $\abv$ be the weak solutions to  the Cauchy problems
  \begin{equation}
\label{abstract-problem-n}
\begin{aligned}
&
\begin{cases}
\frac{\dd }{\dd t}\abv_n(t) + \abopn{n} \abv_n(t)   \ni \EEE \abfon{n}(t) \qquad \text{ in }  \absp_n \EEE \  \  \foraa t \in (0,T),
\\
\abv_n(0) = \abv^0_n,
\end{cases}
\\
&   \begin{cases}
\frac{\dd }{\dd t}\abv(t) + \abop \abv(t)    \ni \EEE \abfo(t) \qquad \text{ in } \absp \ \  \foraa t \in (0,T),
\\
\abv(0) = \abv^0\,.
\end{cases}
\end{aligned}
\end{equation}
 If
 \begin{equation}
 \label{data-converg}
  \lim_{n\to\infty} |\projn n (\abv^0) -\abv^0_n |_n =0 \quad \text{and} \quad  \lim_{n\to\infty} \int_0^T |\projn n (\abfo(t)) - \abfon n(t) |_n \dd t  =0
 \end{equation}
 and if  for every $\lambda \geq 0$ and $\abv \in \absp$ we have that
 \begin{equation}
 \label{trotter-hyp}
 \text{the sequence } ((\mathsf{I}+\lambda \abopn n )^{-1}( \projn n  (\abv)))_n \text{ converge in the sense of Trotter to } (\mathsf{I}+\lambda \abop )^{-1} (\abv) \text{ as } n\to\infty, \end{equation}
 (where we denote by the same symbol the identity operators $\mathsf{I}: \absp_n \to \absp_n$ and
 $\mathsf{I}: \absp \to \absp$),
 then  $(\abv_n)_n$ converge to $\abv$ in the sense of Trotter uniformly on $[0,T]$, namely \EEE
 %\footnote{\RRE We have also added  the second convergence, in \eqref{trotter-thesis},  by taking into account the statement of \cite[Thm.\ 4.2]{LLOO}, where the (uniform w.r.t.\ time) convergence of the norms is also given. Is it a general result?? We are asking because we were not able to retrieve the convergence of the norms in the  statement of \cite[Thm. 5]{ILM09}... }
 \begin{equation}
 \label{trotter-thesis}
 \lim_{n\to\infty}\sup_{t\in [0,T]}   |\projn n (\mathsf{X(t)}) - \mathsf{X}_n (t)|_n  = 0.
 % \qquad  \sup_{t\in [0,T]}  \left| \|\mathsf{X}_n(t)\|_n - \|\mathsf{X}(t)\| \right| \to 0 \text{ as } n \to\infty.
 \end{equation}
 \end{theorem}
\subsubsection{Setting up  Trotter's theory for our problem}
\label{sss:4.1.2}
In what follows, we establish the setup in which we shall  apply Thm.\ \ref{thm:Trotter}.
%\footnote{\LIC Here, for better readability we have recapitulated the spaces $\absp_n$ etc.}  %We postpone to
We consider the  Hilbert \EEE spaces $\absp_n$ from \eqref{Hilbert-spaces-n},
while the `limiting' Hilbert space $\absp$ is given by \eqref{limiting-space},  with  the inner product  from \eqref{limit-inner-prod}.
Now,  in order to apply Thm.\ \ref{thm:Trotter} \EEE
 we  have to introduce a linear continuous operator $\projn n : \absp \to \absp_n$  that
%`compares' the elements in $\absp$ and in $\absp_n$, as in \eqref{props-Pn}.
 with any element $\abv \in \absp$ associates  a suitable representative $\projn n(\abv) \in \absp_n$. \EEE
  Therefore, the operator
%\RRE specificare quali sono gli spazi $\absp_n$??? \EEE
$\projn n : \absp \to \absp_n$ shall have the form
\begin{equation}
\label{opProjn}
\projn n  (\abv) = \projn n ((\vo{u},\vb{u}); (\vo{v}, \vb{v}) )= (\projn{n}^u(\vo u, \vb u); \projn{n}^v (\vo v, \vb v))\,.
\end{equation}
The choice for the operator  $\projn n^v: \mathsf{V} \to  L^2 (\Omega;\R^3)  $  (with $\mathsf{V}$ from \eqref{limiting-space-3}) \EEE
is guided by the idea  of describing the limiting state in terms of two velocity fields, namely we set
\begin{equation}
\label{projn-v}
 \projn{n}^v (\vo v, \vb v): = \chi_{\Omega_{\eps_n}} \vo v + (1{-}\chi_{\Omega_{\eps_n}}) (\sop{\eps_n})^{-1} [\vb v]\,.
\end{equation}
%\RRE with $\zeta_{\eps_n}$...???? \EEE
The choice for $\projn{n}^u$, specified in  \eqref{projn-u} below,  reflects how a field like $\vb u$ may  be involved \EEE in the asymptotic behavior of $(u_n)_n$. Indeed,
 first of all we consider  the unique function  $\vb{u}^n$  satisfying
 %\footnote{\PERME Christian,  in your handwritten notes you wrote
%$ \vb{u}^n (\hat{x}, \pm \eps_n) =  \vb{u}^{\pm}(\hat x, 0) $ but it should be  $ \vb{u}^n (\hat{x}, \pm \eps_n) =  \vo{u}^{\pm}(\hat x, 0) $, right?}
\begin{equation}
\label{NEW-elliptic-phin}
\begin{gathered}
\text{$\vb{u}^n \in H^1(B_{\eps_n};\R^3)$ with $ \vb{u}^n (\hat{x}, \pm \eps_n) =   \vo{u}^{\pm}(\hat x, 0) $   for a.a.\ $\hat x \in S$
  %  \gamma_S (\vo{u}^{\pm})$
 and  }
\\
 \int_{B_{\eps_n}} \rmD W_{\lambda_n,\mu_n} (e(\vb{u}^n)) \cdot e(\varphi) \dd x   = \int_{B}  \rmD W_{\bar\lambda,\bar\mu} (\partial_3 \vb{u}  {\symprod} e_3) \cdot e(\eps_n,\sop{\eps_n}[\varphi]) \dd x
 \quad \text{for all } \varphi \in H_{S_{\eps_n}^+ {\cup} S_{\eps_n}^-}^1(B_{\eps_n};\R^3)\,,
\end{gathered}
\end{equation}
where $H_{S_{\eps_n}^+ {\cup} S_{\eps_n}^-}^1(B_{\eps_n};\R^3)$ denotes the closed subspace of $H^1(B_{\eps_n};\R^3)$ consisting of the functions with null trace on   $S_{\eps_n}^+ {\cup} S_{\eps_n}^-$.  From the functions $\vb{u}^n$ we then derive functions defined on $B = S \times (-1,1)$ by resorting to the operator
$\sop{\eps_n}$. Namely, we set
\begin{equation}
\label{UBN}
\ubn: = \sop{\eps_n} [\vb{u}^n].
\end{equation}
From \eqref{NEW-elliptic-phin} we deduce that the functions $\ubn$ fulfill  (recall that $ \gamma_{S^\pm}$ denote the traces on $S^\pm = S \times \{\pm 1\}$, cf.\ \eqref{s-pm}) \EEE
%\footnote{\PERME Christian, there was a typo in your notes: you wrote
% $\int_\Omega  \rmD W_{\lambda_n,\mu_n} (e(\eps_n,\ubn)).....$, bu it should be $\int_B$, right?}
\begin{equation}
\label{NEW-ubn}
\begin{gathered}
\text{$\ubn \in H^1(B;\R^3)$ with $ \gamma_{S^\pm}(\ubn)= \gamma_{S^\pm}(\vb{u}) $
 and  }
\\
\frac1{\eps_n}   \int_{B}  \rmD W_{\lambda_n,\mu_n} (e(\eps_n,\ubn)) \cdot e(\eps_n,\psi) \dd x   = \int_{B}  \rmD W_{\bar\lambda,\bar\mu} (\partial_3 \vb{u}  {\symprod} e_3) \cdot e(\eps_n,\psi) \dd x
 \quad \text{for all } \psi \in H_{S^+ {\cup} S^-}^1(B;\R^3)\,.
\end{gathered}
\end{equation}
 The functions $\ubn$ will enter into the definition of $\projn{n}^u$. Before specifying in which way, however, let us gain further insight into the properties of the sequence $(\ubn)_n$ in the  following result, where we are using the notation $\wubn$ for the first two components of the function $\ubn$, cf.\ Notation \ref{notation-initial}.
 \EEE
\begin{lemma}
\label{lemma4.1}
%Assume \RRE richiamare delle ipotesi???? \EEE
The following properties hold:
\begin{enumerate}
\item the sequence   $(\wubn)_n$ \EEE
% \RRE  non capisco questa notazione, \`e analoga a quella $\widehat\xi$ per un vettore $\xi \in \R^3$, cf.\ Notation \ref{notation-initial} ?? \EEE
 converge weakly   to $\widehat{u}_B$   in $L^2(B;\R^2) $;
 \item the sequence  $((\ubn)_3, e(\eps_n, \ubn))_n$ \EEE
% \RRE  non capisco questa notazione ?? \EEE
 converge strongly to $(u_{B3}, \partial_3 u_B {\symprod} e_3)$ in $H_{\partial_3} (B) \times L^2(B;\Sym)$  (with the space
 $H_{\partial_3} (B)$ defined analogously as $H_{\partial_3} (B;\R^3)$, cf.\ \eqref{space-H-partial3}); \EEE
 \item  if $ u_B$ belongs to $H^1(B;\R^3),$ then  $(\ubn)_n$  \EEE converge strongly to $u_B$ in
  $L^2(B;\R^3).$ \EEE
 %$H_{\partial_3} (B;\R^3) $.
\end{enumerate}
\end{lemma}
%\RRE Qui bisognerebbe aggiungere la `Note 4' dei suoi appunti, pag.\ 12, per  giustificare questo lemma. \EEE
\begin{proof}
 Let $\vb{u}^* \in H^1(B;\R^3)$ fulfill $\gamma_{S^\pm}(\vb{u}^*) =\gamma_{S^\pm}(\vb{u})  $. We plug in \eqref{NEW-ubn} the test function $\psi = \ubn-\vb{u}^*$, thus obtaining
\begin{equation}
\label{new-test}
\frac1{\eps_n} \int_{B} \rmD W_{\lambda_n,\mu_n} (e(\eps_n, \ubn)) \cdot e(\eps_n, \ubn-\vb{u}^*) \dd x = \int_B \rmD W_{\bar\lambda,\bar\mu}(\partial_3 \vb{u}{\symprod} e_3) \cdot e(\eps_n, \ubn-\vb{u}^*) \dd x.
\end{equation}
Combining the estimate
$\tfrac1{\eps_n} \rmD W_{\lambda_n\mu_n}(e) \cdot e\geq \tfrac{\mu_n}{\eps_n} |e|^2$  (cf.\eqref{ellittico}) with the fact that
 $\frac{\mu_n}{\eps_n} \to \bar \mu >0$ by Hypothesis \ref{hyp:params}, from \eqref{new-test}
 we easily deduce that
 \[
 \exists\, C>0 \ \forall\, n \in \N\, : \qquad
 \frac{\bar{\mu}}2 \int_{B} |e(\eps_n, \ubn)|^2 \dd  x \leq C(1{+} \|e(\eps_n, \ubn)\|_{L^2(B;\R_{\mathrm{sym}}^3)}).
 \]
This implies that the sequence $(e(\eps_n, \ubn))_n$ is bounded in $L^2(B;\R_{\mathrm{sym}}^3)$.
  We will use this to conclude that $(\ubn)_n$ is bounded in $L^2(B;\R^3)$ via a Poincar\'e-type estimate, namely
\begin{equation}
\label{added-general-f}
\exists\, C>0 \ \forall\, n \in \N  \ \forall\, z \in H_{S^+{\cup}S^-}^1(B;\R^3)\, :   \qquad \int_{B} |z|^2 \dd x \leq C \int_{B} |e(\eps_n, z)|^2 \dd x \,.
\end{equation}
We deduce  \eqref{added-general-f}
from
 estimate \eqref{e:4.1}, written for
 the function
 \[
 w: = \begin{cases}
 0 &\text{in } \Omega_{\eps_n},
 \\
 \sop{\eps_n}^{-1}[z] &\text{in } B_{\eps_n}
 \end{cases}
 \]
 (observe that $\sop{\eps_n}^{-1}[z] \in H^1(B_{\eps_n};\R^3)$ with
 $\gamma_{S_{\eps_n}^{\pm}}( \sop{\eps_n}^{-1}[z]) =0$).
Then,
\[
\begin{aligned}
 \int_{B} |z|^2 \dd x =  \int_{B} |\sop{\eps_n}[ \sop{\eps_n}^{-1}[z]]|^2 \dd x &  = \frac1{\eps_n} \int_{B_{\eps_n}} |\sop{\eps_n}^{-1}[z]|^2 \dd x
 \\ & \stackrel{(1)}{\leq}  \frac1{\eps_n} \left( 2\eps_n \int_{S_{\eps_n}} |\sop{\eps_n}^{-1}[z]|^2 \dd x
 + C\eps_n^2 \int_{B_{\eps_n}} |e(\sop{\eps_n}^{-1}[z])|^2 \dd x  +0 \right)
\\
&\stackrel{(2)}{=}   C\eps_n  \int_{B_{\eps_n}} |e(\sop{\eps_n}^{-1}[z])|^2 \dd x
\\
& \stackrel{(3)}{=}  C \int_{B} |e(\eps_n,  \sop{\eps_n}[ \sop{\eps_n}^{-1}[z]])|^2 \dd x
 = C\int_B |e(\eps_n,z)|^2 \dd x,
 \end{aligned}
\]
where (1) follows from \eqref{e:4.1}, (2)  from the fact that $\gamma_{S_{\eps_n}^{\pm}}( \sop{\eps_n}^{-1}[z]) =0$,  and (3) from \eqref{newlabel}.
Next, choosing $z = \ubn{-}\vb{u}^*$ in \eqref{added-general-f} we conclude that
%\PERME NON CAPISCO... dare pi\`u dettagli...
\begin{equation}
\label{Poincare}
\int_{B} |\ubn{-}\vb{u}^*|^2 \dd x \leq C \int_{B} |e(\eps_n, \ubn{-}\vb{u}^*)|^2 \dd x.
\end{equation}
Since  $(e(\eps_n, \ubn))_n$ is bounded in $L^2(B;\R_{\mathrm{sym}}^3)$, we infer that
 $(\ubn)_n$ is bounded in $L^2(B;\R^3)$. Therefore,
 up to a (not relabelled) subsequence, the functions $(\ubn,e(\eps_n,\ubn))_n$ weakly converge to a pair
 $(\vb{\bar u}, \partial_3 \vb{\bar u}{\symprod}e_3)$, where the identification of the weak limit of $(e(\eps_n,\ubn))_n$
 follows from a distributional convergence argument. We are then in a position to pass to the limit in \eqref{NEW-ubn} and thus deduce that the
 function $\vb{\bar u}$ fulfills
 \begin{equation}
\label{lim-eq-ubn}
\begin{gathered}
\text{$\vb{\bar u} \in H_{\partial_3}^1(B;\R^3)$ with $ \gamma_{S^\pm}(\vb{\bar u})= \gamma_{S^\pm}(\vb{u}) $
 and  for all } \psi \in H_{\partial_3}^1(B;\R^3) \text{ with } \gamma_{S^\pm} (\psi) =0 \text{ there holds}
\\
\int_{B}  \rmD W_{\bar\lambda,\bar\mu} (\partial_3 \vb{\bar u}{\symprod}e_3) \cdot  (\partial_3 \psi{\symprod}e_3)   \dd x   = \int_{B}  \rmD W_{\bar\lambda,\bar\mu} (\partial_3 \vb{u}  {\symprod} e_3) \cdot (\partial_3 \psi {\symprod}e_3) \dd x\,.
\end{gathered}
\end{equation}
Therefore,
 $\partial_3 (\vb{\bar u} - \vb u) =0$ and, since  $ \gamma_{S^\pm}(\vb{\bar u})= \gamma_{S^\pm}(\vb{u}) $, we ultimately have
 that $\vb{\bar u} = \vb u$.   Having uniquely identified the limit we eventually gain convergence along the \emph{whole} sequence $(\eps_n)_n$. We thus conclude claim (1).
 \par
 Eventually, \eqref{new-test} implies
\[
\begin{aligned}
\lim_{n\to\infty} \int_B  2 W_{\bar\lambda,\bar\mu} (e(\eps_n, \ubn))  \dd x
 & =
\lim_{n\to\infty} \int_B \rmD W_{\bar\lambda,\bar\mu} (e(\eps_n, \ubn)) \cdot e(\eps_n, \ubn) \dd x  \\ &   = \lim_{n\to\infty}\frac1{\eps_n}
\int_B  \rmD W_{\lambda_n,\mu_n} (e(\eps_n, \ubn)) \cdot e(\eps_n, \ubn) \dd x\\ &  = \int_B \rmD W_{\bar\lambda,\bar\mu} (\partial_3 \vb{u} {\symprod} e_3 ) \cdot (\partial_3 \vb{u} {\symprod} e_3 ) \dd x
\\ &  = \int_B 2 W_{\bar\lambda,\bar\mu} (\partial_3 \vb{u}{\symprod} e_3 ) \dd x\,.
\end{aligned}
\]
Since the functional $q\mapsto \left( \int_B  W_{\bar\lambda,\bar\mu} (q)  \dd x \right)^{1/2}$ induces a norm  equivalent   to the usual one  on $L^2(B;\R^{3\times 3}_{\mathrm{sym}})$, we thus deduce that $e(\eps_n, \ubn)$ converge to $\partial_3 \vb{u} \symprod e_3$ strongly in $L^2(B;\R^{3\times 3}_{\mathrm{sym}})$
and then
 $(\ubn)_3\to (\vb{u})_3$ strongly  in $H_{\partial_3}^1(B)$.  This gives claim (2).

Finally, suppose that $\vb{u} \in H^1(B;\R^3)$. The analogue of
\eqref{Poincare}, i.e.
$
\int_{B} |\ubn{-}\vb{u}|^2 \dd x \leq C \int_{B} |e(\eps_n, \ubn{-}\vb{u})|^2 \dd x,
$
combined with claim (2),
yields that  $\ubn\to \vb{u}$ strongly in $L^2(B;\R^3)$.
%also in view of \eqref{Poincare} we infer that $\wubn \to \vb{\widehat u}$ in $L^2(B;\R^2)$ and $(\ubn)_3\to (\vb{u})_3$ in $H_{\partial_3}^1(B)$.
 This concludes the proof.
\end{proof} \EEE

\par
 We define the operator $\projn{n}^u : \mathsf{U} \to H^1_{\Gdir}(\Omega;\R^3) $ (with $\mathsf{U}$ from
\eqref{limiting-space-2})
 by %\footnote{\PERME See the comment file}
\begin{equation}
\label{projn-u}
\projn n^u (\vo u,\vb u)(x): = \begin{cases}
(\sop{\eps_n})^{-1} [\ubn](x) = \vb{u}^n(x) & \text{ if } x \in  B_{\eps_n},
\\
\xi(x_3) \vo{u}(\hat{x}, x_3 +\mathrm{sign}(x_3) \eps_n) +(1{-}\xi(x_3)) \vo u(x) & \text{ if } x \in  B_{\eps_0}^{\pm} \setminus B_{\eps_n}^{\pm},
\\
\vo u(x) & \text{ if } x \in  \Omega_{\eps_0},
\end{cases}
\end{equation}
where
$B_{\eps_0}^\pm$,  $B_{\eps_n}^\pm$ are  from \eqref{Beps0-pm} \EEE and
$\xi$ is a function in $\mathrm{C}^\infty_{\mathrm{c}}(\R)$ such that
%\footnote{\PROB Christian, could you please confirm to us that there was a typo in your notes? You meant $\xi(t): =\xi_t$ for $ \frac{\eps_0}3 \leq |t| \leq  \frac{2\eps_0}3$ as we have written in \eqref{cut-off-xi}, and not just for  $\frac{\eps_0}3 \leq t \leq  \frac{2\eps_0}3 $ as you wrote in your handwritten notes, right? Otherwise, the definition of $\xi(t)$ for $ -\frac{2\eps_0}3 \leq t \leq  -\frac{\eps_0}3$ would be missing.... }
\begin{equation}
\label{cut-off-xi}
\xi(r) :=
\begin{cases}
1  & \text{ if }   |r| \leq \frac{\eps_0}3, \EEE
\\
\in [0,1] & \text{ if } \frac{\eps_0}3 \leq |r| \leq  \frac{2\eps_0}3,
\\
0
& \text{ if } |r| \geq \frac{2\eps_0}3
\end{cases}
\end{equation}
 (namely,  for $|r|\in [ \tfrac{\eps_0}3,  \tfrac{2\eps_0}3 ] $ we set $\xi(r): = \xi_r $ with $\xi_r $ some element in $[0,1]$). 
Note that $\projn n^u (\vo u,\vb u)$ does belong to   $H_{\Gdir}^1 (\Omega;\R^3) $ because $u =( \vo u, \vb u)$ belongs to the space $ \Spu $ from \eqref{limiting-space-2}.
In what follows, we will often write $\projn n^u (u)$ in place of $ \projn n^u (\vo u,\vb u)$
for notational simplicity.
We are now in a position to prove the
 % Hence, we  have proved  the
following result.
%\PERME dovremmo dare due parole di commento sulla dimostrazione della Prop.\
%\ref{prop:4.2}...
 \EEE
\begin{proposition}
\label{prop:4.2}
We have that
\begin{enumerate}
\item there exists $C>0$ such that for all $\abv \in \absp$ there holds $|\projn n \abv|_n \leq C|\abv|$;
\item there holds $\lim_{n\to\infty} |\projn n \abv|_n = |\abv|$,
\end{enumerate}
namely properties \eqref{props-Pn} hold.
\end{proposition}
\begin{proof}
We start by recalling that, for every $\abv= (u,v)  \in \absp$ with $ u = (\vo u,\vb  u)$ and $ v = (\vo v, \vb v)$
and with
$\projn n \abv   = (\projn n^u(u), \projn n^v(v)) $, %= ( \projn n^u (\vo u,\vb u),  \projn n^v (\vo v,\vb v))$,
we have that
\[
\begin{aligned}
|\projn n \abv|_n^2  & =   \int_{\Omega_{\eps_n}} \rho^* |\projn n^v(v)|^2 \dd x  + \int_{B_{\eps_n}} \rho_{n} |\projn n^v(v)|^2 \dd x
\\ & \qquad
+ \int_{\Omega_{\eps_n}} a e(\projn n^u(u)) \cdot e(\projn n^u(u)) \dd x +\int_{B_{\eps_n}} \rmD W_{\lambda_n,\mu_n} ( e(\projn n^u(u))) \cdot e(\projn n^u(u)) \dd x
\\ &
\doteq I_{1}^n+I_2^n+I_3^n+I_4^n,
\end{aligned}
\]
(cf.\ \eqref{Hilbert-spaces-n}),
while  by \eqref{ip1} we have
\[
\begin{aligned}
|\abv|^2 &  = \int_{\Omega} \rho^* |\vo v|^2 \dd x + \bar\rho \int_B |\vb v|^2 \dd x + \int_{\Omega{\setminus}S} a e(\vo u) \cdot e(\vo u) \dd x +\int_B \rmD W_{\bar\lambda,\bar\mu} (\partial_3 \vb u{\symprod} e_3)\cdot (\partial_3 \vb u{\symprod} e_3) \dd x
\\ &
\doteq I_1+I_2+I_3+I_4\,.
\end{aligned}
\]
Now, by the definition \eqref{projn-v} of $\projn n^v(v)$ we have that
\begin{equation}
\label{v-integrals-1}
\begin{aligned}
I_1^n+I_2^n &  =  \int_{\Omega_{\eps_n}} \rho^* |\vo v|^2 \dd x  + \int_{B_{\eps_n}} \rho_{n} |\sop{\eps_n}^{-1} [\vb v]|^2 \dd x
\\
 &
\leq   \int_{\Omega} \rho^* |\vo v|^2 \dd x  +\rho_n \eps_n\int_B |\vb v|^2 \dd x
\leq \int_{\Omega} \rho^* |\vo v|^2 \dd x  +(\bar{\rho}+c) \int_B |\vb v|^2 \dd x
\end{aligned}
\end{equation}
 for sufficiently big $n$,
where the last estimate follows from Hyp.\ \ref{hyp:params}(6).
Indeed, by the dominated convergence theorem  and again  Hyp.\ \ref{hyp:params}
we also  have
\begin{equation}
\label{v-integrals-2}
I_1^n + I_2^n \to I_1+I_2\,.
\end{equation}
Further, taking into account that
$\Omega_{\eps_n} = \Omega_{\eps_0} \cup (B_{\eps_0}^+{\setminus}B_{\eps_n}^+) \cup (B_{\eps_0}^-{\setminus}B_{\eps_n}^-) $ and
recalling the definition
 \eqref{projn-u}  of $\projn n^u$, we have
\begin{equation}
\label{v-integrals-3}
\begin{aligned}
I_3^n =  \int_{\Omega_{\eps_0}} a e(\vo{u}) \cdot e(\vo{u}) \dd x  & +
 \int_{B_{\eps_0}^+{\setminus} B_{\eps_n}^+} a e(\xi \vo{u}(\cdot{+}\eps_n e_3) {+}(1{-}\xi) \vo{u}) \cdot e(\xi \vo{u}(\cdot{+}\eps_n e_3) {+}(1{-}\xi) \vo{u})\dd x
 \\ & +
 \int_{B_{\eps_0}^-{\setminus} B_{\eps_n}^-}   a e(\xi \vo{u}(\cdot{+}\eps_n e_3) {-}(1{-}\xi) \vo{u}) \cdot e(\xi \vo{u}(\cdot{-}\eps_n e_3) {+}(1{-}\xi) \vo{u}) \dd x
\end{aligned}
\end{equation}
and it is not difficult to check that, again by the dominated convergence theorem,
\begin{equation}
\label{I3n-quoted-later}
I_3^n \to  \int_{\Omega_{\eps_0}} a e(\vo{u}) \cdot e(\vo{u}) \dd x + \int_{B_{\eps_0}{\setminus}S} a e(\vo{u}) \cdot e(\vo{u}) \dd x =I_3\,.
\end{equation}
Hence, we also have that $I_3^n \leq C I_3$. Finally,  since $\projn n^u(u)= \vb{u}^n$ on $B_{\eps_n}$, we have
\begin{equation}
\label{v-integrals-4}
\begin{aligned}
I_4^n  & = \int_{B_{\eps_n}} \rmD W_{\lambda_n,\mu_n} ( e( \vb{u}^n)) \cdot e( \vb{u}^n) \dd x
\\
 &
 = \frac1{\eps_n} \int_{B} \rmD W_{\lambda_n,\mu_n}(e(\eps_n, \ubn)) \cdot e(\eps_n, \ubn) \dd x
 \\
  &
  = \int_{B} \left( \frac{\lambda_n}{\eps_n} |\mathrm{tr}(e(\eps_n, \ubn))|^2 {+}  \frac{2\mu_n}{\eps_n} |e(\eps_n, \ubn)|^2 \right) \dd x
  \\
   &
  \stackrel{(1)}{\longrightarrow}
  \int_B \left(\bar\lambda |\mathrm{tr}(\partial_3 \vb u {\symprod}e_3)|^2 {+} 2\bar\mu |\partial_3 \vb u {\symprod}e_3|^2  \right) \dd x
  \\
   &
  = \int_B \rmD W_{\bar\lambda,\bar\mu}(\partial_3 \vb u {\symprod}e_3) \cdot (\partial_3 \vb u {\symprod}e_3) \dd x,
 \end{aligned}
\end{equation}
where (1) is due to Hyp.\ \ref{hyp:params}(2) and to the strong convergence $e(\eps_n, \ubn)\to \partial_3 \vb u {\symprod}e_3$  in
$L^2(B;\Sym)$ due to Lemma \ref{lemma4.1}(2).  Clearly, these arguments also give $I_3^n \leq C I_3$.  This concludes the proof.
\end{proof}
\EEE
\par
Even if  the convergence notion from \eqref{4.22}  is the right one \EEE from the mechanical viewpoint, it could be of interest to translate this convergence in terms of some classical conventional convergence notions.
\begin{proposition}
\label{prop4.3}
%Assume \RRE richiamare delle ipotesi???? \EEE
 Let $(\abv_n)_n = (u_n,v_n)_n  $ with $\abv_n \in \absp_n  $ %from \RRE ??? \EEE
  for all $n\in \N$, converge in  the sense of Trotter to some $\abv = (u,v) \in \absp$.
Then, the following convergences hold as $n\to\infty$ %\footnote{\PERME See the comment file.}
\begin{enumerate}
\item  the sequence  $\chi_{\Omega_{\eps_n}} (u_n,e(u_n))$ converge  to
 $(\vo{u}, e(\vo{u}))$ strongly in $L^2(\Omega;\R^3)\times L^2(\Omega{\setminus}S;\R_{\sym}^{3\times3})$; \EEE
%strongly in $L^2(\Omega{\setminus}S;\R^3{\times}\Sym)$;
\item the sequence  $\widehat{\sop{\eps_n}[\chi_{B_{\eps_n}}u_n]} $ \EEE converge to $\widehat{\vb u}$ weakly in $L^2(B;\R^2)$;
\item the sequence   $((\sop{\eps_n}[\chi_{B_{\eps_n}} u_n])_3, e(\eps_n, \sop{\eps_n} [\chi_{B_{\eps_n}}u_n]))_n$ \EEE converge to $(u_{B3}, \partial_3 \vb{u}{\symprod} e_3)$ strongly in  $H_{\partial_3} (B) \times L^2(B;\Sym)$;
\item
moreover, if $\vb u \in H^1(B;\R^3) $, then   $\sop{\eps_n}[\chi_{B_{\eps_n}} u_n] \to \vb u$ \EEE  strongly in $L^2 (B;\R^3)$;
\item
$\chi_{\Omega_{\eps_n}} v_n \to \vo v$ strongly in $L^2(\Omega;\R^3)$;
\item $\sop{\eps_n}[\chi_{B_{\eps_n}} v_n] \to \vb v$ strongly in $L^2(B;\R^3)$.
\end{enumerate}
\end{proposition}
\begin{proof}
%Item (1) was proven in\footnote{\PERME See the comment file.} \cite[Prop.\ 4.3]{LLOO}.
 Item (1) is an immediate consequence of the definition
\eqref{projn-u}
of the operator
 $\projn{n}^u$. \EEE
\par
 As for items (2), (3), (4), the key point is to observe that the convergence in the sense of Trotter of $(u_n,v_n)_n  $  to $(u,v)$
yields that (here,  for simplicity we will write $u_n$ in place of $\chi_{B_{\eps_n}} u_n$)
\[
 \int_{B} \left( \frac{\lambda_n}{\eps_n} |\mathrm{tr}(e(\eps_n,\sop{\eps_n}[\vb{u}^n - u_n]))|^2 {+}  \frac{2\mu_n}{\eps_n}
  |e(\eps_n,\sop{\eps_n}[\vb{u}^n - u_n]))|^2 \right) \dd x  \to 0
\]
(cf.\ the calculations for \eqref{v-integrals-4}), whence
\[
e(\eps_n,\sop{\eps_n}[\vb{u}^n - u_n]) \to 0 \text{ strongly in } L^2(B;\Sym).
\]
Combining this information with the second convergence in Lemma \ref{lemma4.1}(2), we immediately deduce that
$e(\eps_n,\sop{\eps_n}[u_n]\to \partial_3 \vb{u}{\symprod} e_3$ strongly in $L^2(B;\Sym)$, whence the strong convergence of
$(\sop{\eps_n}[u_n])_3$ to $u_{B3}$. This proves item (3).
Next, taking into account \eqref{added-general-f} we also infer that the sequence $(\vb{u}^n - u_n)_n$ is bounded in $L^2(B;\R^3)$, and then items (2) \& (4) follow from items (1) \& (3) in Lemma \ref{lemma4.1}.
\par
As for items (5) \& (6), from the Trotter convergence of $(u_n,v_n)_n  $  to $(u,v)$ we also deduce, in particular, that
\[
\begin{aligned}
&
J_1^n : = \int_{\Omega_{\eps_n}} \rho^* \left( \vo{v}{-}v_n\right)^2 \dd x  \to 0 &&\text{as } n \to\infty,
\\
&
J_2^n : = \int_{B_{\eps_n}}\rho_n \left(\sop{\eps_n}^{-1}[\vb{v}]{-} \chi_{B_{\eps_n}}v_n \right)^2 \dd x  \to 0 &&\text{as } n \to\infty.
\end{aligned}
\]
Now, from $J_1^n\to0$ we immediately deduce item (5); we then observe that
\[
0 = \lim_{n\to\infty}J_2^n  = \lim_{n\to\infty} \int_{B}\rho_n \eps_n\left(\vb{v}{-} \sop{\eps_n}[\chi_{B_{\eps_n}}v_n] \right)^2.
\]
Recalling that $\rho_n\eps_n\to\bar{\rho}>0$ by Hyp.\ \ref{hyp:params}(6), we immediately infer item (6). \EEE
\end{proof}
\subsubsection{Convergence results}
In order to apply Theorem \ref{thm:Trotter} establishing the   convergence in the sense of Trotter (cf.\ \eqref{4.22}) \EEE  of $(\abv_n)_n$ to $\abv$ uniformly on $[0,T]$, it is sufficient to impose suitable conditions on the initial data, which we shall discuss at the end of this section,  and to
check the validity of conditions \eqref{data-converg} and \eqref{trotter-hyp},
 with the operators
  $\abop_n: \absp_n  \rightrightarrows \absp_n$
  %\footnote{\LIC Christian, for better clarity we have recalled the operators $\abop_n$ and the data $\abfo_n$, etc.}
 with domains
\begin{subequations}
\label{opAn}
\begin{equation}
\label{domain-opn}
\begin{aligned}
\rmD(\abop_n) \colon =   \Big\{   \abv = (u,v) \in \absp\, : \
&   (1) \, v \in  H_{\Gdir}^1(\Omega;\R^3)
\\
&
   (2)  \, \exists\, (w,\xi) \in L^2 (\Omega;\R^3)  \times   L^2(B_{\eps_n};\R^3)\EEE \text{ with } \xi \in   \partial \calD(e(v))   \text{ a.e.\ in } B_{\eps_n},
 \text{ s.t. }
   \\
   & \qquad   k_n(w,v')   +  \varphi_n (u,v') + b_n \int_{B_{\eps_n}} \xi \cdot e(v') \dd x =0 \text{ for all } v' \in     H_{\Gdir}^1(\Omega;\R^3)
\Big\}\,,
\end{aligned}
\end{equation}
defined at every $\abv = (u,v)$ by
\begin{equation}
\label{ope-n}
\abop_n \abv : = \left(\begin{array}{cc}
-v
\\
0
\end{array} \right) +  \left\{
 \left(\begin{array}{cc}
0
\\
-w
\end{array} \right) \, : \   \text{$w$ as in  \eqref{domain-opn}(2)} \right\}. \EEE
\end{equation}
\end{subequations}
%It results that $\abop_n$ is a maximal monotone operator.
% \EEE
   and the data
   \begin{equation}
   \label{data-n}
   (\abfo_n)_n , \  (\abv_0^n)_n \quad \text{as in       \eqref{data-Cauchy-abs}.}
   \end{equation}
%written for the approximating data depending on the parameter $n$.
  %$$\abfo_n=$$
  \EEE
Condition \eqref{trotter-hyp}  follows from the following result.
\begin{proposition}
\label{prop:4.4}
%Assume \RRE richiamare delle ipotesi???? \EEE Then,
There holds
\begin{equation}
\label{trotter-statement}
\lim_{n\to\infty} | \projn n( (\mathsf{I}{+}\abop)^{-1}(\Psi))  {-} ((\mathsf{I}{+} \abopn n )^{-1}( \projn n  (\Psi)))|_n =0 \quad \text{for all } \Psi \in \absp.
\end{equation}
\end{proposition}
\noindent
The proof of Proposition \ref{prop:4.4} is postponed after the statement of Proposition
\ref{prop:4.5}, where we are going to check
  \eqref{data-converg}. With this aim, we need to impose
  % which involves  we need to impose
   an additional condition on  the external loading $g$.
   \begin{hypothesis}
   \label{hyp-data}
   We suppose that %$f \in \BV([0,T];L^2(\Omega;\R^3))$ and
    $ g \in  W^{2,\infty}(0,T; L^2(\Gneu;\R^3))$ fulfills \EEE
\begin{equation}
\label{H4}
\begin{aligned}
&
\mathrm{supp}(g) \cap \overline{B}_{\eps_0} = \emptyset \quad \text{for all } t \in [0,T] \text{ and }
\\
&
\text{if } \min\{\mathcal{H}^2(\Gdir^+), \mathcal{H}^2(\Gdir^-)\}=0, \text{ say }  \mathcal{H}^2(\Gdir^-) =0, \text{ then } \mathrm{supp}(g) \cap (\partial\Omega_{\eps_0}^-) = \emptyset\,.
\end{aligned}
\end{equation}
\end{hypothesis}
Observe that \eqref{H4} guarantees
 that the support of $g$ lies outside $\overline{B}_{\eps_0} $ and that, if the lower adhering body is not clamped, then there are no surface forces imposed on its boundary.
  Under the additional Hyp.\ \ref{hyp-data} we shall have the following result, whose proof is postponed to that of Prop.\ \ref{prop:4.4}. \EEE
\begin{proposition}
\label{prop:4.5}
%Assume \RRE richiamare delle ipotesi???? \EEE Then,
There holds
\begin{enumerate}
\item $\lim_{n\to\infty} \int_0^T |\projn n (\abfo(t)) {-} \abfo_n(t)|_n \dd t =0$;
\item $\lim_{n\to\infty} \sup_{t\in [0,T]} |\projn n (\abv^{\mathrm{e}}(t)) {-} \abv^{\mathrm{e}}_n (t)|_n =0$,
\end{enumerate}
 where, according to the decomposition from \eqref{added-label-Gio}, $\abv^{\mathrm{e}}$ and $\abv^{\mathrm{e}}_n$
 are the `stationary' parts of the solutions $\abv$ and
 $\abv_n$. \EEE
 \end{proposition}
\par
Let us now proceed with the \underline{\textbf{proof of Proposition \ref{prop:4.4}}}, which is split into  three \EEE steps. The main idea is to exploit
the characterizations of the resolvents of $\abopn n$ and $\abop$ provided by
Proposition \ref{prop:fromLLOO} and by \eqref{def-resolv-lim}, respectively.
In what follows, we will  consider  a \emph{fixed} element $\Psi = (\psi_1,\psi_2) \in \absp$.
\paragraph{\underline{First step}:} We prove the following
%\footnote{\RRE Elena, da qui fino alla fine della Sez.\ 4 ho capito pochino... per favore, fai anche tu un extra-check, grazie}
\begin{lemma}
\label{lemma:4.2}
For all $w\in \Spu$ there exists a sequence $(w_n)_n\subset H_{\Gdir}^1(\Omega;\R^3)$ such that
\[
\lim_{n\to\infty} \varphi_n(w_n{-} \projn n^u (w),  w_n{-} \projn n^u (w)) =0,
\]
and each term of
\begin{equation}
\label{tilde-Jn}
\widetilde{\mathcal{J}}_n(w_n) : = \frac12 \varphi_n(w_n,w_n) + \frac12 k_n(w_n,w_n) + b_n \int_{B_{\eps_n}} \mathcal{D}(e(w_n)) \dd x + \varphi_n(\projn n^u (\psi_1),w_n)
 -k_n(\projn n^v (\psi_2),w_n)
 \end{equation}
converges to the corresponding term of
\[
\mathcal{J}(w): = \frac12\varphi(w,w) +\frac12 k(w,w) + \int_{B}\overline{\mathcal{D}} (\partial_3 \vb{w}\EEE{\symprod} e_3) \dd x + \varphi(\psi_1,w) - k(\psi_2,w),
\]
 cf.\
\eqref{energy-J-eps-new}.
\end{lemma}
\begin{proof}
Since $\calJ$ is continuous on $\Spu$, it is sufficient to prove the result on a dense subset of $\Spu$, namely  the set
$(H_{\Gdir}^1(\Omega{\setminus}S;\R^3) {\times} H^1(B;\R^3)) \cap \Spu$, and to conclude via a diagonalization argument.
% (cf.\ \RRE  [ATTOUCH].. dare qui un riferimento pi\`u preciso..Elena, potresti per favore controllare?\EEE)
 Then, we set
\[
w_n : = \projn n^u(w)\,.
\]
Now, it follows from Prop.\ \ref{prop:4.2}(2) that $\varphi_n(w_n,w_n)\to \varphi(w,w)$ as $n\to\infty$.
The convergence $ k_n(w_n,w_n)\to  k(w,w) $ stems from
the definition
 \eqref{projn-u}   of $\projn n^u$
 and Lemma \ref{lemma4.1}.
 Indeed,
\[
k_n(w_n,w_n) = \int_{\Omega_{\eps_n}}\rho^* |w_n|^2 \dd x + \int_{B_{\eps_n}} \rho_n |w_n|^2 \dd x\,.
\]
On the one hand, again taking into account that
$\Omega_{\eps_n} = \Omega_{\eps_0} \cup (B_{\eps_0}^+{\setminus}B_{\eps_n}^+) \cup (B_{\eps_0}^-{\setminus}B_{\eps_n}^-) $ and
recalling  \eqref{projn-u}  (cf.\ the arguments for \eqref{v-integrals-3}), we see that
 $ \int_{\Omega_{\eps_n}}\rho^* |w_n|^2 \dd x \to \int_{\Omega}\rho^* |\vo{w}|^2 \dd x $.
 On the other hand,
\[
\int_{B_{\eps_n}} \rho_n |w_n|^2 \dd x = \int_{B_{\eps_n}} \rho_n |\sop{\eps_n}^{-1}[w_{B,n}]|^2 \dd x = \rho_n\eps_n \int_{B} |w_{B,n}|^2 \dd x \to \bar\rho\int_{B}|\vb{w}|^2 \dd x
\]
due to Hyp.\ \ref{hyp:params}(6) and  Lemma \ref{lemma4.1}(3), using that  $\vb w $ belongs to $H^1(B;\R^3)$.
  Analogously, one can pass to the limit in the fifth contribution to
 $\widetilde{\mathcal{J}}_n(w_n)$.
  As for the third term, we have that
 \[
   b_n \int_{B_{\eps_n}} \mathcal{D}(e(w_n)) \dd x
   = \frac{b_n}{\eps_n^{p-1}} \int_{B} \mathcal{D} (e(\eps_n,\sop{\eps_n}[w_n])) \dd x \to \int_B  \overline{\calD}(\partial_3 \vb{w}{\symprod}e_3) \dd x,
   \]
  with $\overline{\calD} $ given by \eqref{dissipation-limit}, taking into account that $e(\eps_n,\sop{\eps_n}[w_n]) \to \partial_3 \vb{w}{\symprod}e_3$ strongly in
  $L^2(B;\Sym)$ by Proposition \ref{prop4.3}(3).
 Finally, to deal with the fourth term we use that
 \[
 \begin{aligned}
 4  \varphi_n(\projn n^u (\psi_1),\projn n^u(w))  & =  \varphi_n(\projn n^u (\psi_1{+}w),\projn n^u(\psi_1{+}w)) -    \varphi_n(\projn n^u (\psi_1{-}w),\projn n^u(\psi_1{-}w))
 \\
 & \longrightarrow \varphi(\psi_1{+}w,\psi_1{+}w) -    \varphi(\psi_1{-}w,\psi_1{-}w) = 4 \varphi(\psi_1,w)\,,
\end{aligned}
 \]
 which concludes the proof.
\EEE
% allows us to conclude the convergence of the first, third, and fourth term of $\widetilde{\mathcal{J}}_n(w_n)$ to the corresponding contributions to $\mathcal{J}(w)$: this can be checked relying on Lemma \ref{lemma4.1}(2)  and on Proposition \ref{prop:4.2}(2).
%Finallu, since  one can pass to the limit in the second and last terms in
% $\widetilde{\mathcal{J}}_n(w_n)$ in view of Lemma
% \ref{lemma4.1}(3) and Proposition  \ref{prop4.3}.
% \PROB
% Christian, we have some doubts about this proof.
% \begin{enumerate}
% \item
% First of all, what do you say, should we add `Note 5' on page 15 of your handwrittem notes?
% \item Moreover, we do not understand how  the convergences from Lemma \ref{lemma4.1}(2)  and Proposition \ref{prop:4.2}(2)
% allow us to take the limit in the single terms of
%$\widetilde{\mathcal{J}}_n$?
% \end{enumerate}
%  \EEE
\end{proof}
\paragraph{\underline{Second step}:} We now show that
\begin{lemma}
\label{lemma:4.3}
Let $\bar{v}_n$ be the (unique) minimizer of   the functional $\widetilde{\mathcal{J}}_n$
from \eqref{tilde-Jn} \EEE
on
% $H_{\Gdir}^1(\Omega{\setminus}S;\R^3)$ \RRE lui scrive
 $H_{\Gdir}^1(\Omega;\R^3)$.
Then, there exists $\bar v = (\vo{\bar v}, \vb{\bar v}) \in \Spu$ such that
\begin{enumerate}
\item the elements $\widetilde{\mathsf{V}}_n :  = (\bar{v}_n,\bar{v}_n)$ converge to $\widetilde{\mathsf{V}}: = (\bar v,\bar v)$  in the sense of Trotter;
%\PERME la notazione con tilde e bar mischiati non \`e il massimo... \EEE
\item every term of $\calJ(\bar v)$ is estimated from above by the $\liminf_{n\to\infty}$ of the corresponding term of $
\widetilde{\mathcal{J}}_n(\bar{v}_n)$;
\item $\bar v $ is the unique minimizer of $\calJ$ on $\Spu$;
\item
$\calJ(\bar v) = \lim_{n\to\infty} \widetilde{\calJ}_n(\bar{v}_n)$ and $|\widetilde{\mathsf{V}}| = \lim_{n\to\infty} |\widetilde{\mathsf{V}}_n |_n$.
\end{enumerate}
\end{lemma}
\begin{proof}
Since
$\widetilde{\mathcal{J}}_n(\bar{v}_n) \leq \widetilde{\mathcal{J}}_n(0) =0 $, the sequence $(\widetilde{\mathsf{V}}_n )_n$ is bounded and there exists
 $\bar v = (\vo{\bar v}, \vb{\bar v})  \in \Spu$ such that,  at least \EEE
along a not relabeled subsequence, there holds %\footnote{\PERME See the comment file.}
\begin{itemize}
\item[(i)] the sequence $\chi_{\Omega_{\eps_n}}(\bar{v}_n, e(\bar{v}_n))$ converges weakly to $(\vo{\bar v}, e(\vo{\bar v}))$ in $L^2(\Omega;\R^3\times \Sym)$;
%\RRE lui ha scritto solo $L^2(B;\R^3)$ ma dimensionalmente non torna \EEE
%%\item[(ii)] $(\widehat{\sop{\eps_n}[\bar{v}_n]})_n$  converges weakly to $\widehat{\vb{\bar v}}$ in $L^2(B;\R^2)$;
%\RRE anche questo  non mi torna, l'operatore $\widehat{\cdot}$  non seleziona forse solo le prime due componenti? allora dovrebbe essere in $L^2(B;\R^2)$, no? \EEE
\item[(ii)] $(\sop{\eps_n}[\bar{v}_n])_3, e(\eps_n, \sop{\eps_n}[\bar{v}_n] )_n$  converges to $(\bar{v}_{B3}, \partial_3 \vb{\bar v}{\symprod} e_3)$  weakly  in $H_{\partial_3}(B) \times L^2(B;\Sym)$;
\item[(iii)]
$\chi_{\Omega_{\eps_n}} \bar{v}_n$ converge to $\vo{\bar v}$ strongly in $L^2(\Omega;\R^3)$;
\item[(iv)] $\sop{\eps_n}[\bar{v}_n]$ converge to $\vb{\bar v}$ weakly in $L^2(B;\R^3)$.
\end{itemize}
Let us only comment on the proof of (iii): from $\int_{\Omega_{\eps_n}} (|e(\bar{v}_n)|^2 +|\bar{v}_n|^2) \dd x \leq C$ and Korn's inequality in
$\Omega_{\eps_0/3}$ we deduce that
the sequence $(\bar{v}_n^{\mathrm{ext}})_n$ defined by $\bar{v}_n^{\mathrm{ext}}(x): = (1-\xi(x_3)) \bar{v}_n(x)$ for all $x\in \Omega$, with $\xi$ the function from \eqref{cut-off-xi}, is bounded in $H^1(\Omega;\R^3)$. 
Hence, 
$\bar{v}_n^{\mathrm{ext}}\to \bar{v}^{\mathrm{ext}}$, with $ \bar{v}^{\mathrm{ext}}(x): =(1-\xi(x_3)) \vo{\bar v}(x) $  strongly in $L^2(\Omega;\R^3)$. Likewise, Korn's inequality in $\Omega^{\pm}$ guarantees that 
the sequence $(\bar{v}_n^*)_n$ with $\bar{v}_n^*(\hat x, x_3): = \xi(x_3 \pm \eps_n) \bar{v}_n (\hat x, x_3\pm \eps_n)$ for all $x\in \Omega^\pm$ is bounded in $H^1(\Omega^\pm;\R^3)$. Hence, $v_n^* \to v^*$  strongly 
 in $L^2(\Omega;\R^3)$, with $v^*(x): = \xi(x_3) \vo{\bar v}(x)$ for all $x\in \Omega$.
  Now, since $v_n^*$ and $\bar{v}_n^{\mathrm{int}}: = \bar{v}_n - \bar{v}_n^{\mathrm{ext}}$ weakly converge to the same limit  in $L^2(\Omega;\R^3)$,  a fortiori we conclude that  
  $ \bar{v}_n =  \bar{v}_n^{\mathrm{ext}}+\bar{v}_n^{\mathrm{int}} \to \vo{\bar v}$ in $L^2(\Omega;\R^3)$.
\par
We use the above convergences for the last two terms contributing to $\widetilde{\mathcal{J}}_n(\bar{v}_n)$,
and an additional classical lower semicontinuity argument for the first three terms
to conclude   Claim (2), at least along a subsequence. Then, from Step $1$ we infer that
$\bar v$ is the unique minimizer in $\Spu$ of the strictly convex functional $\calJ$,
 namely Claim (3). \EEE
 Therefore, the \emph{whole} sequence $(\bar{v}_n)_n$ converge, and
 there holds \EEE
 $\calJ(\bar v ) =\lim_{n\to\infty}
\widetilde{\mathcal{J}}_n(\bar{v}_n) $. 	 In order to complete the proof of   Claim (4), it suffices to observe that
\[
\begin{aligned}
&\limsup_{n\to\infty} \left( \frac12 \varphi_n(\bar{v}_n,\bar{v}_n) {+} \frac12 k_n(\bar{v}_n,\bar{v}_n)   \right)
  \\ & \leq
\limsup_{n\to\infty}
 \left(
\widetilde{\mathcal{J}}_n(\bar{v}_n) {-}  b_n \int_{B_{\eps_n}} \mathcal{D}(e(\bar{v}_n)) \dd x{-} \varphi_n(\projn n^u (\psi_1),\bar{v}_n)
 {+}k_n(\projn n^v (\psi_2),\bar{v}_n) \right)
 \\
 & = \calJ(\bar v )  -\liminf_{n\to\infty}  b_n \int_{B_{\eps_n}} \mathcal{D}(e(w_n)) \dd x -\liminf_{n\to\infty}  \varphi_n(\projn n^u (\psi_1),\bar{v}_n) +
  \lim_{n\to\infty} k_n(\projn n^v (\psi_2),\bar{v}_n)
  \\
  &
  \leq  \calJ(\bar v )  -  \int_{B}\overline{\mathcal{D}} (\partial_3 \vb{\bar{v}}{\symprod} e_3) \dd x  -  \varphi(\psi_1,\bar{v}) + k(\psi_2,\bar{v})
  \\ & = \frac12 \varphi(\bar v,\bar v) +\frac12 k(\bar{v},\bar{v})    \leq \liminf_{n\to\infty} \left( \frac12 \varphi_n(\bar{v}_n,\bar{v}_n) {+} \frac12 k_n(\bar{v}_n,\bar{v}_n)   \right),
  \end{aligned}
\]
which gives the separate convergences
\[
\varphi_n(\bar{v}_n,\bar{v}_n)\to  \varphi(\bar v,\bar v), \qquad  k_n(\bar{v}_n,\bar{v}_n) \to  k(\bar{v},\bar{v})\,.
\]
Hence, we conclude that
\begin{equation}
\label{old-find}
|\widetilde{\mathsf{V}}| = \lim_{n\to\infty} |\widetilde{\mathsf{V}}_n |_n.
\end{equation}
\par
%Since
%\[
%2|\widetilde{\mathsf{V}}_n|_n^2 = \widetilde{\mathcal{J}}_n(\bar{v}_n) - \left( \widetilde{\mathcal{J}}_n(\bar{v}_n) -2 |\widetilde{\mathsf{V}}_n|_n^2\right),
%\]
%one has
%\begin{equation}
%\label{old-find}
%|\widetilde{\mathsf{V}}|^2 = \lim_{n\to\infty} |\widetilde{\mathsf{V}}_n|_n^2.
%\end{equation}
 Then, we use that
\[
|\projn n \widetilde{\mathsf{V}} {-} \widetilde{\mathsf{V}}_n|_n^2 = |\projn n \widetilde{\mathsf{V}} |_n^2 +|\widetilde{\mathsf{V}}_n|_n^2 - 2(\projn n \widetilde{\mathsf{V}},\widetilde{\mathsf{V}}_n)_n\,.
\]
Taking into account \eqref{old-find} and the fact that
 $\lim_{n\to\infty} |\projn n \widetilde{\mathsf{V}} |_n^2 = | \widetilde{\mathsf{V}}|^2$, in order to establish the Trotter convergence of $(\widetilde{\mathsf{V}}_n)_n$ to $ \widetilde{\mathsf{V}} $ it remains to check that  \EEE
\[
\lim_{n\to\infty} (\projn n \widetilde{\mathsf{V}}, \widetilde{\mathsf{V}}_n)_n = |\widetilde{\mathsf{V}}|^2.
\]
 This
 stems from  convergences (i)--(iv).
 In order to check this, we use that
 \[
 \varphi_n(\projn n^u(\bar v), \bar{v}_n) = \int_{\Omega_{\eps_n}} a e (\projn n^u(\bar v)) \cdot e( \bar{v}_n)  \dd x + \int_{B_{\eps_n} }
 \rmD W_{\lambda_n,\mu_n} (e(\vb{\bar v}^n)) \cdot   e( \bar{v}_n)  \dd x \,.
 \]
  	To take the limit in the first term,  we combine the facts that
 $\chi_{\Omega_{\eps_n}} e(\bar{v}_n) \weakto e(\vo{\bar v})$ weakly in $L^2(\Omega;\R_{\sym}^{3\times3})$
 and $e(\projn n^u(\bar v)) \to e (\vo{\bar v})$ strongly in
 $L^2(\Omega{\setminus}S;\R_{\sym}^{3\times3})$ (indeed, the weak convergence
 of $e(\projn n^u(\bar v)) $ to $e (\vo{\bar v})$ in
 $L^2(\Omega{\setminus}S;\R_{\sym}^{3\times3})$ improves to a strong one by the analogue of convergence \eqref{I3n-quoted-later}).
 Then, \EEE
 %\footnote{\PERME Christian, we don't understand why
 %$e(\projn n^u(\bar v)) \to e (\vo{\bar v})$ strongly in
 %$L^2(\Omega{\setminus}S;\R_{\sym}^{3\times3})$. From
 %$
% \varphi_n(\bar{v}_n,\bar{v}_n) \to \varphi(\bar v,\bar v)$ we have that $e(\bar{v}_n) \to e(\vo{\bar v})$ in $L^2(\Omega;\Sym)$. Why do we have
 %$e(\projn n^u(\bar v)) \to e (\vo{\bar v})$?}
 we find that
 \[ \int_{\Omega_{\eps_n}} a e (\projn n^u(\bar v)) \cdot e( \bar{v}_n)  \dd x  \to
 \int_{\Omega{\setminus}S} a e(\vo{\bar v})\cdot e(\vo{\bar v}) \dd x .
 \]
   Let us now show that %\footnote{\PERME See the comment file.}
\begin{equation}
\label{new-proof-DW}
 \int_{B_{\eps_n} }
 \rmD W_{\lambda_n,\mu_n} (e(\projn n^u(\bar v))) \cdot   e( \bar{v}_n)  \dd x  \to  \int_{B}
 \rmD W_{\bar\lambda,\bar\mu} (\partial_3 \vb{\bar v}{\symprod}e_3) \cdot   (\partial_3 \vb{\bar v}{\symprod}e_3)   \dd x\,.
\end{equation}
With this aim, we recall that $\projn n^u(\bar v) = \vb{\bar v}^n$ on $B_{\eps_n}$, with $\vb{\bar v}^n$ the solution to \eqref{NEW-elliptic-phin}. Then, it is sufficient to observe
\[
\begin{aligned}
&
 \int_{B_{\eps_n} } \rmD W_{\lambda_n,\mu_n} (e(\vb{\bar v}^n)) \cdot   e( \bar{v}_n)  \dd x  \\   &   = \frac1{\eps_n} \int_B
  \rmD W_{\lambda_n,\mu_n} (e(\eps_n,\sop{\eps_n}[\vb{\bar v}^n])) \cdot   e(\eps_n,\sop{\eps_n}[ \bar{v}_n])  \dd x
  \\ &
 {=} \frac1{\eps_n} \int_B \left(\lambda_n  \mathrm{tr}(e(\eps_n,\sop{\eps_n}[\vb{\bar v}^n]))  \mathrm{tr}(e(\eps_n,\sop{\eps_n}[{\bar v}_n]))
  +2\mu_n e(\eps_n,\sop{\eps_n}[\vb{\bar v}^n]) \cdot e(\eps_n,\sop{\eps_n}[{\bar v}_n]) \right) \dd  x
  \\
  & \longrightarrow \int_B \left(\bar\lambda |\mathrm{tr}(\partial_3 \vb{\bar v}{\symprod}e_3)|^2{+}2\bar\mu |\partial_3 \vb{\bar v}{\symprod}e_3|^2 \right) \dd x = \int_{B}
 \rmD W_{\bar\lambda,\bar\mu} (\partial_3 \vb{\bar v}{\symprod}e_3) \cdot   (\partial_3 \vb{\bar v}{\symprod}e_3)   \dd x
  \end{aligned}
\]
where the above convergence follows from the fact that $e(\eps_n,\sop{\eps_n}[\vb{\bar v}^n]) \to \partial_3 \vb{\bar v}{\symprod}e_3$
and   $e(\eps_n,\sop{\eps_n}[{\bar v}_n]) \to  \partial_3 \vb{\bar v}{\symprod}e_3$
weakly in $L^2(B;\Sym)$
by Lemmas \ref{lemma4.1}(2) and  convergence (ii) at the beginning of the proof of Lemma \ref{lemma:4.3},  respectively.
%  The very definition of $\vb{\bar v}^n$     (cf.\ \eqref{NEW-elliptic-phin} and \eqref{projn-u})
%  implies that\footnote{\PERME See the comment file.}
%  \[
%   \int_{B_{\eps_n} }
% \rmD W_{\lambda_n,\mu_n} (e(\vb{\bar v}^n)) \cdot   e( \bar{v}_n)  \dd x  \stackrel{(1)}{=}
%  \int_B  \rmD W_{\bar\lambda,\bar\mu} (\partial_3 \vb{\bar v} {\symprod} e_3) \cdot
%e(\eps_n,\sop{\eps_n}[\bar{v}_n]) \dd x \stackrel{(2)}{\to}
%  \int_B  \rmD W_{\bar\lambda,\bar\mu} (\partial_3 \vb{\bar v} {\symprod} e_3) \cdot
% (\partial_3 \vb{\bar v} {\symprod} e_3),
% \]
% where (1) follows from  \eqref{NEW-elliptic-phin} and (2) from the fact that  the functions $e(\eps_n , \sop{\eps_n}[u_n])$ converge  to $\partial_3 \vb u  \symprod e_3$  weakly in   $L^2(B;\R^{3\times 3})$.
  All in all, we conclude that
  $\varphi_n(\projn n^u(\bar v), \bar{v}_n) \to \varphi(\bar v,\bar v)$.
 Finally,  we observe that
 \[
 k_n(\projn n^v({\bar v}), \bar{v}_n) = \int_{\Omega_{\eps_n}} \rho^* \vo{\bar v} \bar{v}_n \dd x + \int_{B_{\eps_n}} \rho_n \sop{\eps_n}^{-1}[\vb{\bar v}] \bar{v}_n \dd x  = \int_\Omega\chi_{\Omega_{\eps_n}} \rho^* \vo{\bar{v}} \bar{v}_n \dd x + \frac{\rho_n}{\eps_n} \int_B \vb{\bar v} \sop{\eps_n}[\bar{v}_n] \dd x\,
 \]
 Then,  convergences (i)  and (iv) stated at the beginning of   the proof 
 % of Lemma
 %\ref{lemma:4.3}
 yield that $ k_n(\projn n^v(\bar v), \bar{v}_n)\to k(\bar v,\bar v)$.
  We have thus established the Trotter convergence of $(\widetilde{\mathsf{V}}_n)_n$ to $ \widetilde{\mathsf{V}} $. % which ultimately, thanks to 
%\footnote{\PROB Even here, Christian, we should probably give more details.. could you please help us with that?}
\end{proof}
 \paragraph{\underline{Third step}:}
   We will show that
 \begin{equation}
 \label{desired-claim}
 \lim_{n\to\infty} |\projn n \widetilde{\abv} - \widetilde{\abv}_n|_n=0 \qquad \text{with } \begin{cases}
 \widetilde{\abv}_n  =  (I+\abop_n)^{-1} (\projn n (\Psi))   %= {\abv}_n  +(\projn n^u(\psi_1), 0),
 \\
  \widetilde{\abv} = (I+\abop)^{-1} (\Psi)  % = \widetilde{\mathsf{V}} + (\psi_1,0)
 \end{cases}
 \end{equation}
 by exploiting the characterizations of the resolvents of $\abop_n$ and $\abop$ from
 \eqref{def-resolv} and \eqref{def-resolv-lim}, respectively. Indeed,
 it follows from  \eqref{def-resolv}  that
\[
 \widetilde{\abv}_n =  (I+\abop_n)^{-1} (\projn n (\psi_1,\psi_2))  = \widetilde{\mathsf{V}}_n  +(\projn n^u(\psi_1), 0),
\]
with $\widetilde{\mathsf{V}}_n = ({\bar v}_, {\bar v}_n)$ the unique minimizer  for the functional   $\widetilde{\mathcal{J}}_n$
from \eqref{tilde-Jn}. Analogously, one has that
 \[
 \widetilde{\abv} = (I+\abop)^{-1} (\psi_1,\psi_2) = \widetilde{\mathsf{V}} + (\psi_1,0)
 \]
 with $ \widetilde{\mathsf{V}} =(\bar v,\bar v)$ and $\bar v$ the unique minimizer of the functional $\calJ$ from \eqref{energy-J-eps-new}.
 Then,
 \[
 \begin{aligned}
 |\projn n \widetilde{\abv} - \widetilde{\abv}_n|_n  & = \left| \projn  n\left (\widetilde{\mathsf{V}} + (\psi_1,0)\right) -  \widetilde{\mathsf{V}}_n -  (\projn n^u(\psi_1), 0) \right|_n\\ & = \left| \projn n (\widetilde{\mathsf{V}}) + (\projn n^u(\psi_1), 0)  - \widetilde{\mathsf{V}}_n -  (\projn n^u(\psi_1), 0) \right|_n =
  \left| \projn n (\widetilde{\mathsf{V}}) - \widetilde{\mathsf{V}}_n \right|_n\to 0
  \end{aligned}
 \]
 as $n\to\infty$ due to Lemma \ref{lemma:4.3}(1).
 Hence, \eqref{desired-claim}
follows. \EEE
%and
%\[
% \widetilde{\abv}_n = {\abv}_n  +(\projn n^u(\psi_1), 0),
%\]
%one has %\footnote{\RRE anche qui, dare piu' dettagli... }
%$\lim_{n\to\infty} |\projn n \widetilde{\abv} - \widetilde{\abv}_n|_n=0$.  \EEE
This concludes the proof of Proposition \ref{prop:4.4}. \QED
\par
Let us now carry out the \underline{\bf proof of Proposition \ref{prop:4.5}}:   As
\[
\frac{f}{\gamma_n} = \chi_{\Omega_{\eps_n}} \frac f{\rho^*} +(1{-}\chi_{\Omega_{\eps_n}} )\frac f{\rho_n}
\]
and
\[
\projn n^v \left(\frac{f}{\rho^*},0\right)  =  \chi_{\Omega_{\eps_n}} \frac f{\rho^*},
\]
we have that \EEE
%\footnote{\PERME See the comment file.}
\[
 \left| {\projn n^v \left(\frac{f}{\rho^*},0\right) \EEE} -  \left(\frac{f}{\gamma_n},0\right)  \right|_n^2 = \int_{B_{\eps_n}}  \frac1{\rho_n}|f|^2 \dd x,
\] the first claim is  a
consequence of the dominated convergence theorem. As remarked in \cite{LLOO}, condition \eqref{H4} implies that $g\mapsto u_n^{\mathrm{e}}$ is a linear mapping with
\[
\varphi_n(u_n^{\mathrm{e}}(t),u_n^{\mathrm{e}}(t)) \leq C|g(t)|_{L^2(\Gneu;\R^3)} \quad \text{for all } t \in [0,T].
\]
Since $u^{\mathrm{e}}(t)$ and $u_n^{\mathrm{e}}(t)$  are minimizers of $\frac12 \varphi(\cdot,\cdot) - L(t)(\cdot)$ and $\frac12 \varphi_n(\cdot,\cdot) - L(t)(\cdot)$, respectively,   \eqref{ass-f&g} and \EEE
the arguments from the proof of  Prop.\ \ref{prop:4.4}  give that
\[
\lim_{n\to\infty}|\projn n \abv^{\mathrm{e}}(t) {-} \abv_n^{\mathrm{e}}(t) |_n = \lim_{n\to\infty} \left | \projn n\left( \frac{\dd \abv^{\mathrm{e}}}{\dd t} (t)\right) {-}  \frac{\dd \abv_n^{\mathrm{e}}}{\dd t}(t) \right |_n =0 \quad \foraa\, t \in (0,T).
\]
But \eqref{ass-f&g}  implies that  the sequence $(\projn n \abv^{\mathrm{e}} {-} \abv_n^{\mathrm{e}})_n$  is bounded \EEE $W^{2,\infty} (0,T)$. Then, we have sufficient compactness to   establish the uniform convergence in \EEE Claim (2) of Prop.\  \ref{prop:4.5}. This concludes its proof. \QED
\par
We conclude
this section by specifying the additional assumption on the initial data under which we will be able to state our convergence result.
\begin{hypothesis}
\label{hyp:further-data}
 We assume that
\begin{equation}
\label{ass-init-data-final}
\exists\, \abv^0 \in \abv^{\mathrm{e}}(0) + \mathrm{D}( \abop); \ \ \abv_n^0 \in \abv_n^{\mathrm{e}}(0) +\mathrm{D}( \abop_n) \text{ and } \lim_{n\to\infty} |\projn n (\abv^0) -\abv_n^0|_n=0\,.
\end{equation}
\end{hypothesis}
Observe that
the first condition imposes a sort of compatibility between the initial
state and the initial loading conditions. The second requirement is a convergence condition that, because of Proposition \ref{prop:4.4}, is for instance
 satisfied by %\PERME NON CAPISCO QUESTA FORMULA \EEE
\[
 \abv_n^0 =  \abv_n^{\mathrm{e}}(0) + (\mathsf{I}{+}\abop_n)^{-1}\projn n (\mathsf{I}{+}\abop)^{-1} (\abv^0{-}\abv^{\mathrm{e}}(0)  ).
\]
\par
We are now in a position to apply the nonlinear Trotter-type Thm.\ \ref{thm:Trotter}  to investigate the asymptotic behavior of the solutions
$(\abv_n^{\mathrm{r}})_n \subset W^{1,\infty} (0,T;\absp_n)$ to the Cauchy problems
\begin{equation}
\label{abstract-problem-rn}
\begin{cases}
\frac{\dd }{\dd t}\abv(t) + \abop_n \abv(t)  \ni  \abfo_n(t) \qquad \text{ in } \absp_n \ \  \foraa\, t \in (0,T),
\\
\abv(0) = \abv_{0}^n -  \abv^{\mathrm{e}}(0),
\end{cases}
\end{equation}
in the Hilbert spaces $(\absp_n,|\cdot|_n)$  from \eqref{Hilbert-spaces-n}, with the operators $(\abop_n)_n$ from  \eqref{opAn} \EEE and the data $\abfo_n$ from \eqref{data-n}.
%\eqref{data-Cauchy-abs}  written for $n$\EEE.\footnote{\RRE se a pagina \pageref{H4} abbiamo scritto esplicitamente  $\abop_n$ e $\abfo_n$, conviene richiamare qui le formule..}
%\RRE 	Qui non si capisce nulla.. anche se lui non vuole, a un certo punto riscriveremo  gli operatori $\abop_n$ e i dati $\abfo_n$ e iniziali esplicitando la loro dipendenza da $n$, cosi' come abbiamo fatto per  \eqref{Hilbert-spaces-n}.. su questi cambiamenti ci dobbiamo imporre perche' senno' non si capisce nulla, sei d'accordo \EEE
Therefore, we deduce that the sequence $(\abv_n^{\mathrm{r}})_n$ converges uniformly, in the sense of Trotter, to the  solution  $\abv^{\mathrm{r}}$
to  the Cauchy problem \eqref{lim-abseq}, with the initial datum $\abv_0^{\mathrm{r}} = \abv_0 - \abv^{\mathrm{e}}(0)$.
 This is summarized in the following theorem, which is the \underline{\textbf{main result of the paper}}.
\begin{theorem}
\label{thm:main}
%\RRE Assume... mi sono persa fra  le varie ipotesi... \EEE
Assume  Hypotheses \ref{hyp:params}, \ref{hyp:dissipation}, \ref{hyp-data},   and \ref{hyp:further-data}. \EEE
Then, the solutions $(\abv_n)_n$ to  the Cauchy problems
% \RRE non capisco pi\`u! Secondo me, a destra nella prima riga di
%\eqref{Cauchy-n}
% non c'\`e
%$(0,f/\gamma)$ come ha scritto lui ma $\abfo_n $ che sarebbe $(-\tfrac{\dd}{\dd t}u_{\mathrm{e}}^n,  \tfrac{f}{\gamma_n})$... lo stesso vale per
%\eqref{Cauchy-infty}???
% \EEE
\begin{equation}
\label{Cauchy-n}
\begin{cases}
\frac{\dd}{\dd t} \abv_n + \abop_n(\abv_n {-} \abv_n^{\mathrm{e}}) \ni  \left(0,\frac{f}{\gamma_n} \right) \EEE \quad \aein\, (0,T),
\\
\abv_n(0) = \abv_n^0
\end{cases}
\end{equation}
converge to the solution $\abv$ of the Cauchy problem
\begin{equation}
\label{Cauchy-infty}
\begin{cases}
\frac{\dd}{\dd t} \abv + \abop(\abv {-} \abv^{\mathrm{e}}) \ni (0,f/\bar{\rho}) \quad \aein\, (0,T),
\\
\abv(0) = \abv^0,
\end{cases}
\end{equation}
in the sense that
\begin{equation}
\label{convergence-projections}
\lim_{n\to\infty} |\projn n (\abv(t)) {-} \abv_n(t)|_n =0 \qquad \text{ uniformly in $[0,T]$}
\end{equation}
 and, in addition, there holds
 \begin{equation}
\label{further-convergence}
\lim_{n\to\infty}  |\abv_n(t)|_n = |\abv(t)| \qquad \text{ uniformly in $[0,T]$.}
\end{equation}
\end{theorem}
\noindent
While convergence \eqref{convergence-projections} is guaranteed by Thm.\ \ref{thm:Trotter}, cf.\ \eqref{trotter-thesis}, the additional \eqref{further-convergence} follows from observing that the sequence $( |\abv_n(\cdot)|_n {-} |\abv(\cdot)| )_n$ is bounded in $W^{1,\infty}(0,T)$, and therefore equicontinuous. This turns the pointwise convergence to $0$ into uniform convergence on $[0,T]$.
\par
Finally,  let us highlight that,  in view of Proposition \ref{prop:4.5}(2), from Thm.\ \ref{thm:main}  we  also infer the uniform convergence of the sequence $\abv_n = \abv_n^{\mathrm{r}}+\abv_n^{\mathrm{e}}$. \EEE
%%%%%%
\section{Conclusive results and remarks}
\label{s:5}
%\RRE qui aggiungeremo pi\`u testo... \EEE
%\PROB Christian, here we should comment on the fact that for some of the systems below we just write the strong
%formulation.. but the systems are just solved in a weak sense, right? \EEE
\paragraph{\bf The variational formulation corresponding to  \eqref{Cauchy-infty}.}
%First of all, let us
 In this final section \EEE we
gain further insight into the variational formulation of the
initial-boundary value problem encompassed in the Cauchy problem \eqref{Cauchy-infty}.
We will distinguish the cases $\bar b<\infty$ and $\bar b =\infty$.
\par
When \underline{$\bar b$ is finite}, a more explicit way of writing \eqref{Cauchy-infty}
 is \
 %RRE mi sembra che il primo integrale, quello del termine inerziale, si possa considerare su $\Omega$ invece di $\Omega\setminus S$.. \EEE
 \begin{equation}
 \label{5.1}
 \begin{aligned}
 &
 \exists\, \xi \in \partial \overline{\calD} (\partial_3 \vb{v}{\symprod} e_3) \text{ such that }
 \\
 & \quad \int_{\Omega} \rho^* \partial_{tt} \vo{u}  \cdot \vo{\psi} \dd x + \int_{\Omega{\setminus}S} a e(\vo{u}) \cdot e(\vo{\psi}) \dd x + \int_B \bar\rho
\partial_{tt}  \vb{u} \cdot \vb{\psi} \dd x + \int_B \left( \mathrm{D} W_{\bar\lambda,\bar\mu} (\partial_3 \vb{u} {\symprod} e_3){+}\xi\right) \cdot (\partial_3\vb{\psi} {\symprod}e_3) \dd x
  \\
  &\quad  =  \int_{\Omega} f \cdot \vo{\psi} \dd x +\int_{\Gneu} g \cdot \vo{\psi} \dd \mathcal{H}^2(x)
  \\
  & \qquad \qquad \text{for all } \psi \in \{ (\vo{\psi},\vb{\psi})\in H_{\Gdir}^1(\Omega{\setminus}S;\R^3){\times}H_{\partial_3}(B;\R^3)\, : \ \gamma_{S^\pm}(\vb{\psi}) =
   \gamma_S(\vo{\psi}^{\pm})\}, \qquad \text{in } (0,T),
 \end{aligned}
 \end{equation}
 supplemented with suitable initial conditions.
 Hence, the limiting behavior may be described in terms of a \emph{coupled} system of two  evolutionary, or  \emph{transient}, \EEE
 problems set in $\Omega\setminus S$ \emph{and} in $B$. Clearly, the stress  and the displacement fields $\vo{\sigma}$ and $\vo{u}$ in the limiting adhering bodies that occupy $\Omega^+$ and $\Omega^-$ satisfy the following relations,
 written in strong form %\footnote{\PROB Christian, could you please check that the last of \eqref{5.2} is written correctly?}
 \begin{equation}
 \label{5.2}
 \begin{aligned}
 &
 \rho^* \EEE \partial_{tt} \vo{u} -  \mathrm{div}\vo{\sigma} =  f %= \rho^* \frac{\dd^2 \vo{u}}{\dd t^2}
&& \text{in } (\Omega{\setminus} S)\times (0,T),
 \\
  &
 \vo{\sigma} = ae(\vo u) && \text{in } (\Omega{\setminus} S)\times (0,T),
 \\
 &
\vo u =0 && \text{on } \Gdir\times (0,T),
\\
&
\vo{\sigma} n = g && \text{on } \Gneu\times (0,T),
\\
&
 \mp( \vo{\sigma}^{\pm}e_3)(\hat x, t)= \frac12 \int_{-1}^1  \Big[   \bar\rho (1{\pm} x_3)  \partial_{tt} \vb{u}(\hat x, x_3, t)
 \\
 & \hspace{3.9cm} {\pm} \left( \left( \left(
\rmD W_{\bar\lambda,\bar\mu}
 (\partial_3 \vb{u}(\hat x, x_3, t){\symprod}e_3) +\xi\right)e_3\right) \cdot e_3 \right) e_3 \Big]\dd x_3  && \text{on } S \times (0,T).
 \end{aligned}
 \end{equation}
 This corresponds to the transient response to the loading $(f,g) $ of each adhering body clamped on $\Gdir^{\pm}$ and linked through a mechanical constraint along $S$.
 Differently from the case of an adhesive layer with a vanishing total mass, which was considered in \cite{LLOO}, the contact between the bodies need not be described only in terms of the traces $\gamma_S(\vo{u}^\pm)$, $\gamma_S(\vo{v}^\pm)$ of the displacement and velocity of the sole adhering bodies.  In fact, one has to consider the additional  variables
 $(\vb u, \vb v = \partial_t \vb u)$ which keep track of the dynamics of the adhesive layer. These variables
 fulfill the following equations
% %\footnote{\PERME Christian, we would expect
% \[
% \bar\rho \partial_{tt}  \vb{u}-\partial_3 (\vb{\sigma}e_3) =0
% \]
% in place of
% \[
% \bar\rho \partial_{tt}  \vb{u}+\partial_3 (\vb{\sigma}e_3) =0
% \]
% }
  \begin{equation}
 \label{5.3}
 \begin{aligned}
  &
 \bar\rho \partial_{tt}  \vb{u}- \partial_3 (\vb{\sigma}e_3) =0 \EEE &&  \text{in } B\times (0,T),
 %\text{in } \Omega\setminus S,
 \\
 &
 \vb{\sigma} \in \rmD W_{\bar\lambda,\bar\mu}(\partial_3 \vb{u}{\symprod} e_3) +  \partial\overline{\calD}(\partial_3 \vb{u}{\symprod}e_3) \EEE && \text{in  } B\times (0,T),
 \\
 &
\gamma_{S^\pm} (\vb u) =\gamma_S(\vo{u}^\pm) && \text{on } S\times (0,T).% \text{ on } \Gdir,
 \end{aligned}
 \end{equation}
  Systems \eqref{5.2} and \eqref{5.3} are \EEE supplemented by suitable initial conditions.
 Such equations are of the same type as those in the  original \EEE layer. Of course,
 the variables
 $(\vb u,\vb v)$ may be eliminated and, consequently (see \eqref{5.2}), the
 contact condition along $S$ between the two adhering bodies is a nonlocal - in time, only- relation between the stress vector
 $\vo{\sigma}^\pm(\hat x, t)) e_3$  \EEE at the current time $t$,
 and the history of $\gamma_S(\vo{u}^\pm(\hat x,\tau))$, with $\tau \in [0,t]$.
  Finally, from the last line of \eqref{5.2}  we deduce that
 \[
  (\vo{\sigma}^{-}e_3)(\hat x, t) - (\vo{\sigma}^{+}e_3)(\hat x, t) =  \int_{-1}^1    \bar\rho   \partial_{tt} \vb{u}(\hat x, x_3, t) \dd x_3,
 \]
 which reflects the fact that the jump of the stress vector on the adhering bodies balances the limiting inertial forces stemming from the adhesive. \EEE
 \par
 In the case \underline{$\bar b =\infty$}, the system reads
 %\footnote{\PERME Christian, in your handwritten notes you suggested that we should add something here (see your Note 52), but we don't understand what you have written, could you please explain it to us?}
 \begin{equation}
 \label{5.1bis}
 \begin{aligned}
 &
\partial_3 \vb v=0 \text{ and  }
 \\
 & \int_{\Omega} \rho^* \partial_{tt} \vo{u} \psi \dd x + \int_{\Omega{\setminus}S} a e(\vo{u}) \cdot e(\psi) \dd x +
 2\bar\rho
 \int_S
  \partial_t \vo{v}\psi \dd x   =  \int_{\Omega} f \cdot \psi \dd x +\int_{\Gneu} g \cdot \psi \dd \mathcal{H}^2(x)
 \end{aligned}
 \end{equation}
 for all $\psi \in H_{\Gdir}^1(\Omega;\R^3)$, in $(0,T)$,
  again supplemented by suitable initial conditions.  Indeed, from $\partial_3 \vb v=0 $, supposing that the initial datum for $\vb u$ is independent of the variable $x_3$ we deduce that $\partial_3 \vb u=0$, and hence that $\JUMP{\vo u}=0$. Hence the space for the test functions in \eqref{5.1bis} is
  $H_{\Gdir}^1(\Omega;\R^3)$. Let us stress that if  {$\bar b =\infty$ then \EEE
the relative motion along $S$ is frozen.
 \paragraph{\bf Other  relative behavior \EEE of the parameters $(\lambda_n,\mu_n)$.}
 As previously mentioned, the analysis in Section \ref{s:4} has been carried out confining the discussion to the case in which the parameters $(\bar\lambda,\bar\mu) $ are in $[0,\infty)\times (0,\infty)$. Let us conclude the paper by
 examining the singular cases
 \begin{enumerate}
 \item $(\bar\lambda,\bar\mu)\in \{\infty\}\times (0,\infty)$;
 \item $(\bar\lambda,\bar\mu)\in [0,\infty) \times \{\infty\}$;
 \item $(\bar\lambda,\bar\mu)\in (0,\infty)\times \{0\}$;
 \item $\bar\lambda=\bar\mu =0$.
 \end{enumerate}
In each of these cases, we will explicitly illustrate spaces $\absp,\, \Spu,\, \Spv$, the bilinear forms $\varphi$ and $k$,
 the operator $\abop$, and the function $u^{\mathrm{e}}$, like we have done
for the case $(\bar\lambda,\bar\mu) \in [0,\infty) \times (0,\infty)$.  We will not  give the proof of the convergence result, as it is a straightforward adaptation
of that developed throughout Sec.\ \ref{ss:4.2}, and we will leave it to the interested reader.
 Nonetheless, we will hint at the main point underlying the identification of the limit problem,
namely the correct identification of the space $\Spu$ which, in turn, will be based
on the analysis of the asymptotic behavior of a sequence $(u_n)_n$ with
$\sup_n \varphi_n(u_n,u_n) <\infty$, cf.\ the arguments in Section \ref{ss:4.1}. \EEE
\par
Let us start by specifying that in each of the above cases
we will have
\[
\begin{aligned}
&
\absp = \Spu \times \Spv,
\\
& \Spv = L^2(\Omega;\R^3) \times L^2(B;\R^3),
\\
&
(\abv,\abv') =\varphi(u,u') + k(v,v') \quad \text{for all } \abv = (u,v), \, \abv'= (u',v') \in \absp,
\end{aligned}
\]
with
 the bilinear form $k$ given by \eqref{ip3},
 %\footnote{\PERME what is $k$???} .... ?? \EEE
 the form $\varphi$ specified  in \eqref{new-phi-1}, \eqref{new-phi-2}, \eqref{new-phi-3}, and \eqref{new-phi-4} below,
 the space $\Spu$ specified  in \eqref{Spu1}, \eqref{Spu2}, \eqref{Spu3}, \eqref{Spu4} below,
 the operator $\abop: \absp \rightrightarrows \absp$  with domain \eee
\begin{subequations}
\label{same-as-415}
\begin{align}
&
\label{domain-oper-final}
\begin{aligned}
\rmD(\abop) = \{ \abv \in \absp\, :  \  & \text{(i) } v \in \calP,
\\
& \text{(ii) } \exists\, (w,\xi) \in \Spv \times \partial\overline{\calD} (\partial_3 \vb{v} {\symprod} e_3) \text{ s.t. }
\\
& \qquad \qquad \qquad k(w,\psi) + \varphi(u,\psi) +\int_B \xi \cdot (\partial_3 \vb{\psi} {\symprod} e_3) \dd x =0 \quad \text{for all }   \psi \in \overline{\Spu}\}, \EEE
\end{aligned}
 \intertext{and defined by}
&
\label{same415b}
\abop \abv : = \left(\begin{array}{cc}
-v
\\
0
\end{array} \right) +  \left\{
 \left(\begin{array}{cc}
0
\\
-w
\end{array} \right)  \, : \  \text{$w$ as in  \eqref{domain-oper-final}(2)} \right\}. \EEE
\end{align}
\end{subequations}
 The space
$ \overline{\Spu}$
featuring in  \eqref{domain-oper-final}(2) will be specified for each of the
singular cases considered (cf.\
   \eqref{Spv1}, \eqref{Spv2}, \eqref{Spv3}, \eqref{Spv4} below);  analogously,
$\tilde v$ shall be  built from $v\in \calP$ \EEE in a way depending on the case under consideration (cf.\ \eqref{tildev1}, \eqref{tildev2}, \eqref{tildev3}, \eqref{tildev4} below),  with
the space
\[
\calP = \{ \psi =(\vo{\psi},\vb{\psi}) \in H^1(\Omega{\setminus}S;\R^3) \times H_{\partial_3} (B;\R^3) \, : \ \gamma_{S^\pm} (\vb \psi) = \gamma_S(\vo{\psi}^\pm) \}\,.
\]
%%%%%
 \paragraph{\bf Case $1$: $(\bar\lambda,\bar\mu)  \in \{\infty\} \times (0,\infty)$.} In this case we have (recall the notation $\JUMP{z} = \gamma_{S}(z^+) - \gamma_{S}(z^-)$)
 \begin{subequations}
 \label{case1}
 \begin{align}
 &
 \label{Spu1}
 \Spu = \{ u = (\vo u,\vb u) \in H_{\Gdir}^1(\Omega{\setminus}S;\R^3) \times H_{\partial_3} (B;\R^3)\, :  \ \JUMP{\vo u}_3=0,  \ \partial_3 u_{B3} =0 \},
 \\
 &
 \label{new-phi-1}
 \varphi(u,u') = \int_{\Omega{\setminus}S} a e (\vo u) \cdot e(\vo{u}') \dd x +\bar\mu \int_B \partial_3\vb{\widehat{u}} \cdot \partial_3 \vb{\widehat{u}}' \dd x,
 \\
  &
  \label{Spv1}
   \overline{\Spu} = \Spu, \EEE
  \\
  &
  \label{tildev1}
  \tilde v = v\,.
 \end{align}
 \end{subequations}
  Indeed, the condition $\frac{\lambda_n}{\eps_n} \to \infty$ implies that, along a sequence $(u_n)_n$ with
 $\sup_n \varphi_n(u_n,u_n) <\infty$, there holds $\mathrm{tr}(e(\eps_n,\sop{\eps_n}[u_n])) \to 0$ in $L^2(B)$, whence the condition  $  \partial_3 u_{B3} =0$
 and, consequently,
 $ \JUMP{\vo u}_3=0 $, as encompassed in the space $\Spu$ from \eqref{Spu1}. \EEE
Here,
 the adhering bodies are in bilateral contact along $S$ and   the tangential component of the stress vector applied along $S$ is given by
 \[
 \widehat{\gamma_S(\vo{\sigma}^{\pm}) e_3} = \pm \frac12 \int_{-1}^1 \left(\bar\rho (1{\pm}x_3) \partial_{tt}\vb{\widehat u} \pm (\bar\mu\partial_3 \vb{\widehat u}{+} \widehat{\xi e_3} )\right) \dd x_3 \qquad \text{on } S\times (0,T),
 \]
 with
 \[
 \left\{
 \begin{array}{lll}
 \displaystyle & \!  \! \! \! \! \! \! \! \xi \in \partial\overline{\calD}(\partial_3 \vb{v} {\symprod} e_3) & \qquad   \qquad  \text{in } B\times (0,T),
 \\
  \displaystyle
 &
 \!  \! \! \! \! \! \! \!  \bar\rho  \partial_{tt} \vb{\widehat u}  - 2\bar\mu \frac{\partial^2 \vb{\widehat u}}{\partial {x_3}^2} - \frac{\partial}{\partial {x_3}}  \widehat{\xi e_3} =0 &  \qquad   \qquad  \text{in } B\times (0,T),
 \\
  \displaystyle
 &
 \!  \! \! \! \! \! \! \!  \widehat{\gamma_{S^\pm}(\vb u)} = \widehat{\gamma_S (\vo{u}^\pm)} &  \qquad   \qquad \text{on } S\times (0,T),
% \\
  %\displaystyle
% &
% \!  \! \! \! \! \! \! \!  \widehat{u}_B(0) = \widehat{\vb{u}^0}, \qquad \frac{\dd \vb{\widehat u}}{\dd t}(0) =
%  \widehat{\vb{v}^0} &   \qquad  \qquad \text{in } B.
 \end{array}
 \right.
 \]
  supplemented by suitable initial conditions. 
 Observe that only the tangential component of the traces on $S$ of the displacement in the adhering bodies is nonlocal-in-time. This is a kind of viscoelastic behavior with long memory when $\bar b$ is finite. When $\bar b =\infty$, since $\partial_3 \vb v=0$, the relative sliding along $S$ is frozen.
 \paragraph{\bf Case $2$:  $(\bar\lambda,\bar\mu)\in [0,\infty) \times \{\infty\}$.} We have
  \begin{subequations}
 \label{case2}
 \begin{align}
 &
 \label{Spu2}
 \Spu = H_{\Gdir}^1(\Omega;\R^3), % \{ u = \vo u \in  \},
 \\
 &
 \label{new-phi-2}
 \varphi(u,u') = \int_{\Omega{\setminus}S} a e (\vo u) \cdot e(\vo{u}') \dd x,
 \\
  &
  \label{Spv2}
  \overline{\Spu}  \EEE= \{ v \in \calP\, : \ \partial_3 \vb v =0, \ \JUMP{\vo v} =0\},
  \\
  &
  \label{tildev2}
  \tilde v = \vo{v}\,.
 \end{align}
 \end{subequations}
  In this case, the condition $\frac{\mu_n}{\eps_n} \to \infty$ implies that $e(\eps_n,\sop{\eps_n}[u_n]) \to 0$ in
 $L^2(B; \R_{\sym}^{3\times3})$, so that $  \partial_3 u_{B} =0$ and, consequently, $\JUMP{\vo u}=0$.  \EEE
 %???????? MA QUESTO NON E' QUELLO CHE E' SCRITTO NELLA \eqref{Spu2}. \EEE
 Hence $\vo u$ satisfies
 \[
 \int_\Omega \rho^* \partial_{tt} \vo{u} \cdot \psi \dd x + \int_S 2\bar\rho   \partial_{tt} \vo{u} \cdot \psi \dd \hat{x} +
 \int_\Omega a e(\vo u) \cdot e(\psi) \dd x = \int_\Omega f \cdot \psi \dd x + \int_{\Gneu} g \cdot \psi \dd \mathcal{H}^2(x)
 \]
 for all $\psi \in H_{\Gdir}^1 (\Omega;\R^3)$, in $(0,T)$.
 The adhering bodies are perfectly bonded and the obtained deformable body is submitted to  surface forces on $S$ corresponding to the limit
 of the vertical forces in the adhesive.
 \paragraph{\bf Case $3$: $(\bar\lambda,\bar\mu)\in (0,\infty)\times \{0\}$.}  Here we find
   \begin{subequations}
 \label{case3}
 \begin{align}
 &
 \label{Spu3}
 \Spu = \{ u = (\vo u, u_{B3}) \in H_{\Gdir}^1(\Omega{\setminus}S;\R^3) \times H_{\partial_3}(B)\, : \ (\gamma_{S^\pm}(\vb u))_3 = (\gamma_S (\vo{u}^\pm))_3 \},
 \\
 &
 \label{new-phi-3}
 \varphi(u,u') = \int_{\Omega{\setminus}S} a e (\vo u) \cdot e(\vo{u}') \dd x+\bar\lambda \int_B \partial_3 u_{B3} \cdot \partial_3 u_{B3}' \dd x,
 \\
  &
  \label{Spv3}
  \overline{\Spu} =\calP, \EEE
  \\
  &
  \label{tildev3}
  \tilde v = (\vo{v}, v_{B3})\,.
 \end{align}
 \end{subequations}
  In this case, from $\varphi_n(u_n,u_n) \leq C$ we may only infer the information that  $(\mathrm{tr}(e(\eps_n,\sop{\eps_n}[u_n])))_n$ is bounded
 in $L^2(B)$, and thus only $u_{B3} $ is defined and belongs to $H_{\partial_3}(B)$, with the condition $ (\gamma_{S^\pm}(\vb u))_3 = (\gamma_S (\vo{u}^\pm))_3$. \EEE
 Every adhering body is subjected to surface forces along $S$ given by
  \[
\gamma_S (\vo{\sigma}^{\pm}) e_3 = \pm \frac12 \int_{-1}^1 \left(\bar\rho (1{\pm}x_3) \partial_t \left(\vb{\widehat v},  \partial_t u_{B3}\right)  \pm (\bar\lambda\partial_3 u_{B3} e_3{+} \xi e_3 )\right) \dd x_3 \quad \text{on } S \times (0,T),
 \]
 with
 \[
 \left\{
 \begin{array}{lll}
 & \displaystyle  \!  \! \! \! \! \! \! \!
 \bar\rho     \partial_t(\vb{\widehat v},  \partial_t u_{B3}) -\frac{\partial}{\partial x_3} (\vb \sigma e_3) =0 & \qquad \qquad \text{in } B\times (0,T),
 \\
 & \displaystyle  \!  \! \! \! \! \! \! \!
\vb\sigma \in \bar\lambda \partial_3 u_{B3} e_3 \symprod e_3 +\partial\overline{\calD}(\partial_3 \vb v {\symprod} e_3) & \qquad \qquad \text{in } B\times (0,T),
 \\
 & \displaystyle \!  \! \! \! \! \! \! \!
(\gamma_{S^\pm}(\vb u))_3 = (\gamma_S(\vo{u}^\pm))_3, \quad \widehat{\gamma_{S^\pm}(\vb v)} = \widehat{\gamma_S(\vo{u}^\pm)} & \qquad \qquad \text{on } S\times (0,T),
 %\\
% & \displaystyle  \!  \! \! \! \! \! \! \!
% u_{B3}(0) = u_{B3}^0, \qquad \frac{\dd u_{B3}}{\dd t}(0) =v_{B3}^0, \qquad \vb{\widehat v}(0) = v_{B3}^0 &\qquad \qquad \text{in } B.
 \end{array}
 \right.
 \]
 supplemented by suitable initial conditions. \EEE
 In this case, the forces are a  nonlocal-in-time function (of the traces on $S$ of the displacements of both adhering bodies)  of viscoelastic with long-memory type.
  \paragraph{\bf Case $4$: $\bar\lambda =\bar\mu=0$.}   Here
    \begin{subequations}
 \label{case4}
 \begin{align}
 &
 \label{Spu4}
 \Spu = \{ u = \vo u  \in H_{\Gdir}^1(\Omega{\setminus}S;\R^3)  \},
 \\
 &
 \label{new-phi-4}
 \varphi(u,u') = \int_{\Omega{\setminus}S} a e (\vo u) \cdot e(\vo{u}') \dd x,
 \\
  &
  \label{Spv4}
  \overline{\Spu} =\calP,
  \\
  &
  \label{tildev4}
  \tilde v = \vo{v}\,.
 \end{align}
 \end{subequations}
 	 Here, the estimate  $\varphi_n(u_n,u_n) \leq C$ yields no information. That is why, the only conditions are given on
 $\vo{v}$ and
  $\vb{v}$; they are involved in the domain of the operator $\abop$ (cf.\ \eqref{domain-oper-final}) and specified by the space
 $ \overline{\Spu}$ from  \eqref{Spv4}. \EEE
 Each adhering body is subject to surface forces along $S$ given by
   \[
 \gamma_S(\vo{\sigma}^\pm) e_3 = \pm \frac12 \int_{-1}^1 \left(\bar\rho (1{\pm}x_3) \partial_t \vb{v}  {\pm} \xi e_3 \right) \dd x_3 \qquad \text{on } S\times (0,T),
 \]
 with
 \[
 \left\{
 \begin{array}{lll}
 &
 \displaystyle  \!  \! \! \! \! \! \! \!
 \bar\rho    \partial_t \vb{v}  -\frac{\partial}{\partial x_3} (\xi e_3) =0 &  \qquad \qquad \text{in } B \times (0,T),
 \\
 &
  \displaystyle  \!  \! \! \! \! \! \! \!
\xi  \in \EEE \partial\overline{\calD}(\partial_3 \vb v {\symprod} e_3)  &  \qquad \qquad \text{in } B \times (0,T),
 \\
 &
  \displaystyle  \!  \! \! \! \! \! \! \!
 \gamma_{S^\pm}(\vb v) = \gamma_S(\vo{v}^\pm)  \EEE &  \qquad \qquad \text{on } S \times (0,T),
% \\
 %&
  %\displaystyle  \!  \! \! \! \! \! \! \!
 %\vb{v}(0) = v_{B}^0  &  \qquad \qquad \text{in } B.
 \end{array}
 \right.
 \]
  supplemented by a  suitable initial condition. \EEE
 These forces are a nonlocal-in-time function (of the traces on $S$ of  the velocity in both adhering bodies) of viscous with long-memory type.


\begin{thebibliography}{99}
%%%%%%%%%%%%%%%%%%%%%%%%%%%%%%%%%%%%%%%%%%%%%%%%
%%%%%%%%%%%%%%%%%%%%%%%%%%%%%%%%%%%%%%%%%%%%%%%%



\bibitem{BBLR1}
Bonetti E., Bonfanti G., Lebon F., Rizzoni R.,
\emph{A model of imperfect interface with damage},
Meccanica, 52 (8), 1911--1922, (2017).

\bibitem{BBLR2}
Bonetti E., Bonfanti G., Lebon F.,
\emph {Derivation of imperfect interface models coupling damage and temperature}, Computers and Mathematics with Applications, 11, 2906--2916, (2019).


\bibitem{BBR1}
Bonetti E., Bonfanti G., Rossi R.,
\emph{Global existence
for a contact problem with adhesion},
 Math. Meth. Appl. Sci., 31, 1029--1064, (2008).

	
\bibitem{BBR3}	
Bonetti E., Bonfanti G., Rossi R.,
\emph{Thermal effects in adhesive contact: modelling and analysis},
Nonlinearity, 22, 2967--2731, (2009).
	


%\bibitem{BBRfrictemp}	
%Bonetti E., Bonfanti G., Rossi R.,
%\emph{Analysis of a temperature-dependent model for adhesive contact with friction},
%Phyisca D, 285, 29-42, (2014).


\bibitem{BBRen}
Bonetti E., Bonfanti G., Rossi R.,
\emph{Modeling via internal energy balance and analysis of adhesive contact with friction in thermoviscoelasticity},
Nonlin. Anal. RWA, 22, 473--507, (2015).

\bibitem{Brezis73} H.\ Brezis, Op\'erateurs Maximaux-Monotones et Semi-Groupes de Contraction dans les Espaces de Hilbert, North-Holland, 1972.


\bibitem{DLR14}
Dumont  S., Lebon, F., Rizzoni, R.,
\emph{An  asymptotic  approach  to  the  adhesion  of  thin  stiff  films},
Mech. Res. Com.,  58, 24--35, (2014).

\bibitem{FPRZ}
Freddi L.,  Paroni, R.,  Roub\`{\i}\v{c}ek T.,  Zanini C., \emph{Quasistatic delamination models for Kirchhoff-Love plates}, ZAMM Z. Angew. Math. Mech., 91, 845--865,  (2011).

\bibitem{FRZ}
Freddi L.,    Roub\`{\i}\v{c}ek T.,  Zanini C.,
\emph{Quasistatic delamination of sandwich-like Kirchhoff-Love plates}, J. Elast., 113, 219--250, (2013).

\bibitem{FRE87}
Fr\'emond M.,
\emph{Adh\'erence des solides}, J. Mec. Theor. Appl., 6, 383--407, (1987).
%
\bibitem{FRE01}
Fr\'emond M.,
\emph{Non-Smooth Thermo-mechanics},
Springer (2001).

\bibitem{LR11}
Lebon, F., Rizzoni, R.,
\emph{Asymptotic behavior of a hard thin linear elastic interphase: An energy approach},
Int. J. Sol. Struct., 48, 441--449, (2011).

\bibitem{Licht-Weller}
Licht C., Weller T.,
\emph{Approximation of semi-groups in the sense of Trotter and asymptotic mathematical modeling in physics of continuous media},
Discrete Contin. Dyn. Syst. Ser. S, 12, 1709--1741, (2019).



 \bibitem{KLA91}
 Klarbring A.,
\emph{Derivation of the adhesively bonded joints by the asymptotic
 expansion method},
Int. J. Eng. Sci., 29, 493--512, (1991).

\bibitem{KMR}
Ko\v{c}vara M., Mielke A.,  Roub\'{\i}\v{c}ek T.,
\emph{A rate-independent approach to the delamination problem},
Math. Mech. Solids, 11, 423--447, (2006).



\bibitem{L93}
Licht C.,
\emph{Comportement asymptotique d'une bande dissipative mince de faible rigidit\'e},
% (French. English, French summary) [Asymptotic behavior of a thin dissipative layer with low stiffness]
C. R. Acad. Sci. Paris S\'er. I Math., 317,  429--433, (1993).

\bibitem{LM97}
Licht C., Michaille G.,  \emph{A modelling of elastic bonded joints}, Adv.\ Math.\ Sci.\ Appl.,  7, 711--740, (1997).

\bibitem{LLOO}
C.\ Licht, A.\ L\'eger, S.\ Orankitjaroen, and A.\ Ould Khaoua:
Dynamics of elastic bodies connected by a thin soft viscoelastic
layer.  J. Math. Pures Appl., 99,  685--703, (2013).


\bibitem{mielke}
Mielke A., Roub\'{\i}\v{c}ek T., Thomas M.,
\emph{From damage to delamination in nonlinear elastic materials at small strains},
J. Elast., 109, 235--273, (2012).


\bibitem{ILM09}
O.\ Iosifescu, C.\ Licht, and G.\ Michaille, \emph{Nonlinear boundary conditions in Kirchhoff-Love plate theory}, J. Elast., 96, 57--79, (2009).


\bibitem{RSZ}
 Roub\'{\i}\v{c}ek T., Scardia L., Zanini C.,
\emph{Quasistatic delamination problem},
Contin. Mech. Thermodyn., 21,  223--235, (2009).


\bibitem{Trotter}
 Trotter, H. F. Approximation of semi-groups of operators. Pacific J. Math., 8,  887--919, (1958).


\end{thebibliography}
\end{document}